\documentclass[12pt,a4paper]{amsart}
\allowdisplaybreaks[1]
\usepackage{amstext}
\usepackage{amsmath,amsthm}
\usepackage{amssymb}
\usepackage{mathtools}
\usepackage{tikz}

\usepackage[top=30truemm,bottom=30truemm,left=25truemm,right=25truemm]{geometry}

\DeclareFontEncoding{OT2}{}{}
\DeclareFontSubstitution{OT2}{cmr}{m}{l}

\DeclareFontFamily{OT2}{cmr}{\hyphenchar\font45}
\DeclareFontShape{OT2}{cmr}{m}{l}{%
<5><6><7><8><9>gen*wncyr%
<10><10.95><12><14.4><17.28><20.74><24.88>wncyr10}{}

\DeclareMathAlphabet{\mathcyr}{OT2}{cmr}{m}{l}

\newtheorem{thm}{Theorem}[section]
\newtheorem{lem}[thm]{Lemma}
\newtheorem{prop}[thm]{Proposition}
\newtheorem{cor}[thm]{Corollary}

\theoremstyle{definition}
\newtheorem{defn}[thm]{Definition}

\theoremstyle{remark}
\newtheorem{rem}[thm]{Remark}

\newcommand{\harb}{\mathbin{\overline{\ast}}}
\newcommand{\hart}{\mathbin{\overset{t}{\ast}}}
\newcommand{\sh}{\mathbin{\mathcyr{sh}}}
\newcommand{\sha}{\mathbin{\widetilde{\mathcyr{sh}}}}
\newcommand{\sht}{\mathbin{\overset{t}{\mathcyr{sh}}}}
\newcommand{\ba}{\boldsymbol{a}}
\newcommand{\be}{\boldsymbol{e}}
\newcommand{\bk}{\boldsymbol{k}}

\newcommand{\bm}{\boldsymbol{m}}
\newcommand{\bp}{\boldsymbol{p}}

\newcommand{\bF}{\mathbb{F}}
\newcommand{\bQ}{\mathbb{Q}}
\newcommand{\bR}{\mathbb{R}}
\newcommand{\bZ}{\mathbb{Z}}
\newcommand{\cA}{\mathcal{A}}

\newcommand{\cF}{\mathcal{F}}
\newcommand{\cI}{\mathcal{I}}
\newcommand{\cS}{\mathcal{S}}
\newcommand{\cZ}{\mathcal{Z}}
\newcommand{\fH}{\mathfrak{H}}
\newcommand{\fS}{\mathfrak{S}}
\newcommand{\fZ}{\mathfrak{Z}}

\DeclareMathOperator{\dep}{dep}
\DeclareMathOperator{\wt}{wt}
\DeclareMathOperator{\supp}{supp}

\begin{document}

\title[Yamamoto's interpolation of FMZVs and FMZSVs]{Yamamoto's interpolation of finite multiple zeta and zeta-star values}

\author{Hideki Murahara}
\address[Hideki Murahara]{Nakamura Gakuen University Graduate School,
 5-7-1, Befu, Jonan-ku, Fukuoka, 814-0198, Japan} 
\email{hmurahara@nakamura-u.ac.jp}

\author{Masataka Ono}
\address[Masataka Ono]{Multiple Zeta Research Center, Kyushu University, 744, Motooka, Nishi-ku,
Fukuoka, 819-0395, Japan} 
\email{m-ono@math.kyushu-u.ac.jp}

\thanks{This research was supported in part by JSPS KAKENHI Grant Numbers 16H06336.}

\begin{abstract}
 We study a polynomial interpolation of finite multiple zeta and zeta-star values with variable $t$, 
 which is an analogue of interpolated multiple zeta values introduced by Yamamoto. 
 We introduce several relations among them and, in particular, prove the cyclic sum formula, the Bowman--Bradley type formula, and the weighted sum formula. 
 The harmonic relation, the shuffle relation, the duality relation, and the derivation relation are also presented.  
\end{abstract}

\maketitle

%%%%%%%%%%%%%%%%%%%%%%%%%%%%%%%%%%%%%%%%%%%%%%%%%%%%%%%%%%%%%%%%%%%%%%%%%%%%%%%%%%%%%%%%%%%%%%%%%%%%%%%
\section{Introduction}
%%%%%%%%%%%%%%%%%%%%%%%%%%%%%%%%%%%%%%
\subsection{Interpolated multiple zeta values}
The notion of interpolated multiple zeta values ($t$-MZVs) which was introduced by Yamamoto \cite{Yam13} is an interpolation polynomial of multiple zeta values (MZVs) and multiple zeta-star values (MZSVs) for a fixed index. 

An index is a sequence of positive integers.
For an index $\bk=(k_{1},\ldots,k_{r})$, 
the integer $k\coloneqq k_{1}+\cdots+k_{r}$ is called the weight of $\bk$
(denoted by $\wt(\bk)$) and the integer $r$ is called
the depth of $\bk$ (denoted by $\dep(\bk)$).
Then, for an index $(k_{1},\ldots,k_{r})$ with $k_r \ge2$, the MZVs and the MZSVs are defined by 
\begin{align*}
 \zeta(k_1,\dots, k_r)
 &\coloneqq \sum_{\substack{1\le n_1<\cdots <n_r \\ n_i \in \bZ}} \frac{1}{n_1^{k_1}\cdots n_r^{k_r}} \in \bR, \\
 \zeta^\star (k_1,\dots, k_r)
 &\coloneqq \sum_{\substack{1\le n_1\le \cdots \le n_r \\ n_i \in \bZ}} \frac{1}{n_1^{k_1}\cdots n_r^{k_r}} \in \bR, 
\end{align*} 
and the $t$-MZVs by
\begin{align} \label{tMZV}
 \zeta^t(k_1,\dots,k_r)
 \coloneqq \sum_{\bp} \zeta(\bp) t^{r-\dep(\bp)} \in \bR[t].
\end{align}
The symbol $\sum_{\bp}$ stands for the sum where $\bp$ runs over all indices of the form
$\bp=(k_1\, \square\, \cdots\, \square\, k_r)$. Here, each $\square$ is filled by a comma `$,$' or a plus `$+$'. 
Note that $\zeta^0=\zeta$ and $\zeta^1=\zeta^\star$ hold.
Several algebraic relations among $t$-MZVs are already known 
(see Yamamoto \cite{Yam13}, Tanaka--Wakabayashi \cite{TW16}, Li--Qin \cite{LQ17}, and Li \cite{Li19}).

%%%%%%%%%%%%%%%%%%%%%%%%%%%%%%%%%%%%%%
\subsection{Interpolated finite multiple zeta values}
The main topic of this paper is to consider the counterpart of $t$-MZVs in the field of finite multiple zeta(-star) values (FMZ(S)Vs), 
and introduce several relations among them.

%%%%%%%%%%%%%%%%%%%%%%%%%%%%%%%%%%%%%%
\subsubsection{Finite multiple zeta(-star) values}
Kaneko and Zagier \cite{KZ19} defined two types of FMZ(S)Vs: 
$\mathcal{A}$-multiple zeta(-star) values ($\mathcal{A}$-MZ(S)Vs) and $\cS$-multiple zeta(-star) values ($\cS$-MZ(S)Vs).

Set $\cA\coloneqq \prod_p\bF_p/\bigoplus_p\bF_p$, where $p$ runs over all primes. 
It is known that $\bQ$ is embedded in $\cA$ so $\cA$ becomes a $\bQ$-algebra (see Kaneko \cite{Kan19} and Kaneko--Zagier \cite{KZ19}). For an index $(k_{1},\ldots,k_{r})$, the $\cA$-MZVs and the $\cA$-MZSVs are defined by
\begin{align*}
 \zeta_\cA(k_1,\dots,k_r)
 &\coloneqq \Biggl(\sum_{\substack{1\le m_1<\dots<m_r<p \\ m_i \in \bZ}}\frac{1}{m_1^{k_1}\dotsm m_r^{k_r}}\bmod p\Biggr)_p \in \cA,\\
 \zeta_\cA^{\star}(k_1,\dots,k_r)
 &\coloneqq \Biggl(\sum_{\substack{1\le m_1\le\dots\le m_r<p \\ m_i \in \bZ}}\frac{1}{m_1^{k_1}\dotsm m_r^{k_r}}\bmod p\Biggr)_p \in \cA.
\end{align*}

Let $\cZ$ be the $\bQ$-linear subspace of $\bR$ spanned by $1$ and all MZVs. 
For an index $(k_{1},\ldots,k_{r})$, we define the $\cS$-MZVs by
\begin{align*}
 \zeta_\cS(k_1,\dots,k_r)
 &\coloneqq \sum_{i=0}^{r}(-1)^{k_{i+1}+\dots+k_r}\zeta(k_1,\dots,k_i)
 \zeta(k_r,\dots,k_{i+1})\bmod\zeta(2) \in \cZ/\zeta(2)\cZ,
\end{align*}
where we understand $\zeta(\emptyset)=1$.
Here, the symbol $\zeta$ on the right-hand side means the regularized values coming from harmonic regularization, 
i.e., real values obtained by taking constant terms of harmonic regularization 
as explained in Ihara--Kaneko--Zagier \cite{IKZ06}. 
For an index $(k_{1},\ldots,k_{r})$, we also define the $\cS$-multiple zeta-star values ($\cS$-MZSVs) by 
\begin{align*}
 \zeta_{\cS}^{\star}(k_{1},\dots,k_{r})
 \coloneqq \sum_{\substack{\square\textrm{ is either a comma `,' } \\
 \textrm{ or a plus `+'}
 }
 }\zeta_{\cS}(k_{1}\square\cdots\square k_{r}) \in \cZ/\zeta(2)\cZ.
\end{align*}

Denoting by $\cZ_{\cA}$ the $\bQ$-linear subspace of $\cA$ spanned by $1$ and all $\cA$-MZVs, Kaneko and Zagier conjectured that there is an isomorphism between $\cZ_{\cA}$ and $\cZ/\zeta(2)\cZ$ as $\bQ$-algebras such that $\zeta_{\cA}(k_1,\ldots, k_r)$ and $\zeta_{\cS}(k_1,\ldots, k_r)$ correspond to each other 
(for details, see Kaneko \cite{Kan19} and Kaneko--Zagier \cite{KZ19}). 
In the following, the letter $\mathcal{F}$ stands for either $\mathcal{A}$ or $\mathcal{S}$.

%%%%%%%%%%%%%%%%%%%%%%%%%%%%%%%%%%%%%%
\subsubsection{Interpolation of finite multiple zeta and zeta-star values}
For each index $\bk$, we define the polynomial that interpolates the FMZV $\zeta_{\cF}(\bk)$ and FMZSV $\zeta^{\star}_{\cF}(\bk)$ with variable $t$ ($t$-FMZVs) by
\begin{align*}
 \zeta_{\cF}^t(k_1,\dots,k_r)
 \coloneqq \sum_{\bp}\zeta_{\cF}(\bp) t^{r-\dep(\bp)}, 
\end{align*}
where the variable $\bp$ runs over the same place in (\ref{tMZV}). 
Then we easily see $\zeta^0_{\cF}=\zeta_{\cF}$ and $\zeta^1_{\cF}=\zeta^\star_{\cF}$ hold.

We note that $\zeta^t_{\cA}(\bk)$ was first defined by Seki \cite{Sek17} and he interpolated the sum formulas for $\cA$-MZVs and $\cA$-MZSVs, which were proved by Saito--Wakabayashi \cite{SW15}, partially. 
Remark that Seki proved the formula only for $\zeta^t_{\cA}(\bk)$, but his proof worked for $\zeta^t_{\cS}(\bk)$. 
For a positive integer $k$, let $\fZ_{\cA}(k)\coloneqq ( B_{p-k}/k \bmod{p} )_p$ and $\fZ_{\cS}(k)\coloneqq \zeta(k) \bmod{\zeta(2)}$, where $B_n$ is the $n$-th Seki--Bernoulli number.
\begin{thm}[Sum formula; Seki \cite{Sek17}] \label{sumF}
 For positive integers $k,r$ with $1\le r\le k$, we have
 \begin{align*}
  \sum_{\substack{ k_1+\cdots+k_{r}=k \\ k_1,\dots,k_{r-1}\ge1, k_{r}\ge2 }} 
  \zeta_{\cF}^t (k_1,\dots,k_{r}) 
  &=\sum_{j=0}^{r-1} \left\{\binom{k-1}{j} +(-1)^r\binom{k-1}{r-1-j} \right\}
   t^j (1-t)^{r-1-j} \fZ_{\cF}(k). 
 \end{align*}
\end{thm} 

In this paper, we give several algebraic relations among $t$-FMZVs and, in particular, we prove the following three theorems:
%%%%%
\begin{thm}[Cyclic sum formula] \label{cycsumF}
 For a non-empty index $(k_1,\dots,k_r)$ with $(k_1,\dots,k_r)\ne (\underbrace{1,\dots,1}_{r})$, we have
 \begin{align*}
   &\sum_{l=1}^{r} \sum_{j=1}^{k_l-1}
   \zeta_{\cF}^{t} (j,k_{l+1},\dots,k_r,k_1,\dots, k_{l-1}, k_{l}+1-j) \\
   &=(1-t) \sum_{l=1}^{r} 
    \bigl(\zeta_{\cF}^{t} (k_{l+1},\dots,k_r,k_1,\dots,k_{l-1},k_{l}+1)
     +\zeta_{\cF}^{t} (k_{l+1}+1,k_{l+2},\dots,k_r,k_1,\dots,k_{l}) \bigr) \\
   &\quad +\sum_{l=1}^r\zeta_{\cF}^{t} (1,k_{l+1},\dots,k_r,k_1,\dots,k_{l}). 
 \end{align*}
\end{thm}
\begin{thm}[Bowman--Bradley type formula] \label{BBtype}
 Let $a$ and $b$ be odd positive integers, and $c$ an even positive integer. 
 Then, for any non-negative integers $l$ and $m$ with $(l,m)\ne (0,0)$, we have
 \begin{align*} 
  \sum_{\substack{\sum_{i=0}^{2l}m_i=m \\ m_0, \ldots, m_{2l}\ge0}}
  \zeta^t_{\cF}\bigl(\{c\}^{m_0}, a, \{c\}^{m_1}, b, \{c\}^{m_2}, \ldots, a, \{c\}^{m_{2l-1}}, b, \{c\}^{m_{2l}}\bigr)
  =0.
 \end{align*}
Here, for a positive integer $k$ and a non-negative integer $m$, we write $\{k\}^m\coloneqq \underbrace{k,\ldots ,k}_{m}$.
\end{thm}
\begin{rem}
 We prove a more general theorem in Section 3 (see Theorem \ref{BBtype_main}).
\end{rem}
\begin{thm}[Weighted sum formula] \label{wt_sum_formula}
 For a positive integer $k$ and an odd positive integer $r$ with $1\le r\le k$, we have
 \begin{align*}
  \sum_{\substack{k_{1}+\cdots+k_{r}=k \\ k_1,\dots,k_r\ge1}}
  2^{k_r-1} \zeta_{\cF}^t (k_{1},\ldots,k_{r}) 
  =0.
 \end{align*}
\end{thm}

The contents of this paper is organized as follows: 
in Section 2, we prove Theorem \ref{cycsumF} by reducing it to the cyclic sum formula for FMZVs.
In Section 3, we prove the Bowman--Bradley type formulas (Theorems \ref{BBtype} and \ref{BBtype_main}) by comparing the coefficients of `Bowman--Bradley sum' in the two successive degrees in $t$-FMZVs. 
In Section 4, we partially interpolate the weighted sum formulas obtained by Hirose--Murahara--Saito \cite{HMS19} for $\cA$-MZ(S)Vs and Murahara \cite{Mur18} for both FMZ(S)Vs. 
In Section 5, we introduce several formulas that are relatively easy to obtain, such as the harmonic relation, the shuffle relation, the duality relation, and the derivation relation.

%%%%%%%%%%%%%%%%%%%%%%%%%%%%%%%%%%%%%%%%%%%%%%%%%%%%%%%%%%%%%%%%%%%%%%%%%%%%%%%%%%%%%%%%%%%%%%%%%%%%%%%
\section{Cyclic sum formula}
The cyclic sum formulas for MZVs and MZSVs were proved by Hoffman--Ohno \cite[(1)]{HO03} and Ohno--Wakabayashi \cite[Theorem 1]{OW06}, respectively.  
Yamamoto \cite[Theorem 5.4]{Yam13} interpolated them to $t$-MZVs: 
\begin{align}\label{eq: CSF for tMZV}
 &\sum_{l=1}^r\sum_{j=1}^{k_l-1}\zeta^t(j, k_{l+1}, \ldots, k_r, k_1, \ldots, k_{l-1}, k_l+1-j)\\
 &=(1-t)\sum_{l=1}^r\zeta^t(k_{l+1}+1, \ldots, k_r, k_1, \ldots, k_{l-1}, k_l)+k\zeta(k+1)t^r, \nonumber
\end{align}
where $r, k_1, \ldots, k_r \ge1$, $k_1, \ldots, k_r$ are not all 1 and $k\coloneqq k_1+\cdots+k_r$.
Yamamoto proved \eqref{eq: CSF for tMZV} by reducing it to the cyclic sum formula for MZVs by differentiating with variable $t$. Our proof of Theorem 1.2 differs from his proof by introducing a cyclic star index $C_m$.

On the other hand, the counterparts for $\cA$-MZ(S)Vs and $\cS$-MZ(S)Vs were obtained by Kawasaki--Oyama \cite[Theorem 1.2]{KO19} and Hirose--Sato (unpublished), respectively (see also Hirose--Murahara--Ono \cite[Theorem 2.4]{HMO20}).
Here, we prove a generalization for $t$-FMZVs (Theorem \ref{cycsumF}) by using their results.

For the proof of Theorem \ref{cycsumF}, we prepare some notation.
Let $\cI=\bigoplus_{r=0}^\infty\bQ[\bZ^r_{\ge1}]$ be the $\bQ$-vector space freely spanned by all indices.
We endow the $\bQ$-vector space $\cI$ with a (non-commutative) ring structure by the concatenation product. We extend $\zeta^t_{\cF}$ to the $\bQ[t]$-linear map from $\cI[t]$.
For an index $(k_1, \ldots, k_r)$, we put
 \begin{align*}
  (k_1,\dots,k_r)^t
  \coloneqq \sum_{\substack{\square\textrm{ is either a comma `,' } \\ \textrm{ or a plus `+'}} }
   (k_1\square k_2\square\cdots\square k_r) \times t^{\text{the number of `+'}} \in \cI[t]
 \end{align*}
and call $(k_1,\dots,k_r)^t$ a $t$-index of $(k_1, \ldots, k_r)$. We extend $\bQ$-linearly the definition of the $t$-index to elements in $\cI$ and for an element $I$ in $\cI$, we denote its $t$-index by $I^t$. 

Moreover, for a non-empty index $\bk=(k_1, \ldots, k_r)$ and a non-negative integer $m$ with $m \le r-1$, let $\bk^t_m$ be the coefficient of $\bk^t$ at $t^m$. That is, 
\begin{align*}
\bk^t_m
\coloneqq 
\sum_{\substack{\square\textrm{ is either a comma `,' } \\ \textrm{ or a plus `+'} \\ \textrm{the number of `+'=$m$} }}
(k_1\,\square\,\cdots\,\square\,k_r) \in \cI.
\end{align*}
We set $\bk^t_{-1}\coloneqq 0$. %Moreover, for $I=\sum_{i=1}^n a_i\bk^{(i)} \in \cI$ with $\dep(\bk^{(1)})=\cdots=\dep(\bk^{n})=r$, and $m \in \bZ_{\ge0}$ with $m \le r-1$, we set $I_m\coloneqq \sum_{i=1}^na_i \bk^{(i)}_m$.

For a non-empty index $\bk=(k_{1},\dots,k_{r})$ and $m \in \bZ_{\ge0}$ with $m \le r-1$, we define the cyclic index $C_m(\bk)$ by
 \begin{align*}
  C_m(\bk)
  \coloneqq \sum_{\substack{\square\textrm{ is either a comma `,' } \\ \textrm{ or a plus `+'} \\ \textrm{the number of `+'=$m$} }}
  (k_1\,\square\,\cdots\,\square\,k_r\,\square) \in \cI. 
 \end{align*}
Here, we understand $(k_1\,\square\,\cdots\,\square\,k_r\,+)=(k_r+k_1\,\square\,\cdots\,\square\,k_{r-1} \square)$ and $(k_1\,\square\,\cdots\,\square\,k_r\,, )=(k_1\,\square\,\cdots\,\square\,k_r)$.
We note that all indices appearing in $C_m(\bk)$ have the same depth$(=r-m)$.

Set $\supp C_m(\bk) \coloneqq \{\text{monomials appearing in $C_m(\bk)$}\}$. Finally, for any $I$-valued $(m-r)$-variable function $f$, we denote by $\sum_{(a_1, \ldots, a_{r-m}) \in \supp C_m(\bk)} f(a_1, \ldots, a_{r-m})$ the sum of each monomial appearing in $C_m(\boldsymbol{k})$. For example, since $C_1(k_1,k_2,k_3)=(k_1+k_2,k_3)+(k_1,k_2+k_3)+(k_3+k_1,k_2)$, we have $\sum_{(a_1, a_2) \in \supp C_1(k_1,k_2,k_3)} f(a_1, a_2)=f(k_1+k_2,k_3)+f(k_1,k_2+k_3)+f(k_3+k_1,k_2)$.
Note that $\sum_{(a) \in \supp C_2(k_1,k_2,k_3)} f(a)=3f(k_1+k_2+k_3)$.
\begin{lem} \label{aaaaa}
For a non-negative integer $m$ with $m\le r-1$ and any non-empty indices $\bk=(k_1,\dots,k_r)$ with $(k_1,\dots,k_r)\ne(1,\dots,1)$, we have
\begin{align*}
&\sum_{l=1}^r \sum_{j=1}^{k_l-1}(j, k_{l+1}, \ldots, k_r, k_1, \ldots, k_{l-1}, k_l+1-j)^t_m \\
&=\sum_{(a_1, \ldots, a_{r-m})\in\supp C_m(\bk)}\sum_{i=1}^{r-m}\sum_{j=1}^{a_i-1}
(j, a_{i+1}, \ldots, a_{r-m}, a_1, \ldots, a_{i-1}, a_i+1-j) \nonumber \\
&\quad +\sum_{(a_1, \ldots, a_{r-m}) \in \supp C_m(\bk)} \sum_{i=1}^{r-m}
(a_i, \ldots, a_{r-m}, a_1, \ldots, a_{i-1}, 1)-\sum_{l=1}^r(k_l, \ldots, k_r, k_1, \ldots, k_{l-1}, 1)^t_m. \nonumber 
\end{align*}
\end{lem}
\begin{proof}
The proof is obtained in a similar way to Hirose--Murahara--Ono \cite[Proposition 6.6]{HMO20}. 
Set $\alpha(\boldsymbol{k})\coloneqq \{(k_i,\dots,k_r,k_1,\dots,k_{i-1}) \mid i=1,\dots,r \}$.
We have
\begin{align*}
&\sum_{l=1}^r \sum_{j=1}^{k_l-1}(j, k_{l+1}, \ldots, k_r, k_1, \ldots, k_{l-1}, k_l+1-j)^t_m \\
&=\sum_{l=1}^r \sum_{j=1}^{k_l}(j, k_{l+1}, \ldots, k_r, k_1, \ldots, k_{l-1}, k_l+1-j)^t_m
-\sum_{l=1}^r(k_l, \ldots, k_r, k_1, \ldots, k_{l-1}, 1)^t_m. 
\end{align*}
Since
\begin{align*}
&\sum_{l=1}^r \sum_{j=1}^{k_l}(j, k_{l+1}, \ldots, k_r, k_1, \ldots, k_{l-1}, k_l+1-j)^t_m \\
&=
 \sum_{(k_1',\dots,k_r')\in\alpha(\boldsymbol{k})}
 \sum_{1\le a\le b<r}
 \sum_{j=1}^{k_r'} 
 \sum_{\substack{ \square= \text{`+' or `,'} \\ \text{the number of `+'}=m-r-a+b+1 }} \\
 &\qquad\qquad\qquad\qquad\qquad\qquad\quad 
  (j+k_{1}'+\cdots+k_a',k_{a+1}'\square\cdots\square k_b',k_{b+1}'+\cdots+k_r'+1-j) \\
%%%%%%
&=
 \sum_{(k_1',\dots,k_r')\in\alpha(\boldsymbol{k})}
 \sum_{l=1}^{r}
 \sum_{s=1}^{l}
 \sum_{j=1}^{k_s'} 
 \sum_{\substack{ \square=\text{`+' or `,'} \\ \text{the number of `+'}=m-l+1 }} \\
 &\qquad\qquad\qquad\qquad\qquad\qquad\quad 
  (j+k_{s+1}'+\cdots+k_l',k_{l+1}'\square\cdots\square k_r',k_1'+\cdots+k_s'+1-j) \\
%%%%%%
&= 
 \sum_{(k_1',\dots,k_r')\in\alpha(\boldsymbol{k})}
 \sum_{l=1}^{r}
 \sum_{j=1}^{k_1'+\cdots+k_l'} 
 \sum_{\substack{ \square= \text{`+' or `,'} \\ \text{the number of `+'}=m-l+1 }}
 (j,k_{l+1}'\square\cdots\square k_r',k_1'+\cdots+k_l'+1-j)\\
%%%%%%
&=\sum_{i=1}^{r-m} 
 \sum_{(a_1, \ldots, a_{r-m}) \in \supp C_m(\bk)}
 \sum_{j=1}^{a_i} 
 (j, a_{i+1}, \ldots, a_{r-m}, a_1, \ldots, a_{i-1}, a_i+1-j),  
\end{align*}
we get the result. 
\end{proof}

For a non-empty index $(k_1, \ldots, k_r) \ne (\{1\}^r)$, we define the element $F^t(k_1,\dots,k_r) \in \cI[t]$ corresponding to $ \mathrm{L.H.S.}-\mathrm{R.H.S.}$\ of Theorem \ref{cycsumF}. Set $\cI_1 \coloneqq \bigoplus_{r=0}^{\infty}\bQ[\bZ^r_{\ge1} \setminus (\{1\}^r)]$. We extend $F^t$ to the $\bQ$-linear map from $\cI_1$ to $\cI[t]$.

\begin{prop}\label{prop: key of CSF}
For a non-negative integer $m$ with $m\le r-1$ and any non-empty indices $(k_1,\dots,k_r)$ with $\bk=(k_1,\dots,k_r)\ne(1,\dots,1)$, let $F^t(\bk)_m$ be the coefficient of $F^t(\bk)$ at $t^m$. Then we have %the coefficient of $F^t (k_1,\dots,k_r)$ at $t^m$ is equal to 
\begin{align*}
F^t(\bk)_m=F^0 (C_m(\bk)).
\end{align*}
\end{prop}
\begin{proof}
From \cite[Lemma 6.3]{HMO20}, we have
\begin{multline}\label{eq: 1st of LHS of key}
\sum_{l=1}^r (k_{l+1}, \ldots, k_r, k_1, \ldots, k_{l-1}, k_l+1)^t_m\\
=\sum_{(a_1, \ldots, a_{r-m})\in \supp C_m(\bk)}\sum_{i=1}^{r-m}(a_{i+1}, \ldots, a_{r-m}, a_1, \ldots, a_{i-1}, a_i+1) 
\end{multline}
and 
\begin{multline}\label{eq: 2nd of LHS of key}
\sum_{l=1}^r (k_l+1, k_{l+1}, \ldots, k_r, k_1, \ldots, k_{l-1})^t_m\\
=\sum_{(a_1, \ldots, a_{r-m}) \in \supp C_m(\bk)}\sum_{i=1}^{r-m}(a_i+1, a_{i+1}, \ldots, a_{r-m}, a_1, \ldots, a_{i-1})
\end{multline}
for $0 \le m \le r-1$. 

From \cite[Lemma 6.3]{HMO20}, we see that the coefficient of 
\begin{align*}
\sum_{l=1}^{r}(1, k_{l+1}, \ldots, k_r, k_1, \ldots, k_l)^t
-t\sum_{l=1}^{r}(k_{l+1}+1, k_{l+2}, \ldots, k_r, k_1, \ldots, k_l)^t
\end{align*}
at $t^m$ is equal to
\begin{align}\label{eq: 3rd of LHS of key}
&\sum_{l=1}^{r}\bigl\{(1, k_{l+1}, \ldots, k_r, k_1, \ldots, k_l)^t_m
- (k_{l+1}+1, k_{l+2}, \ldots, k_r, k_1, \ldots, k_l)^t_{m-1}\bigr\}\\
&=\sum_{l=1}^{r}(1, (k_{l+1}, \ldots, k_r, k_1, \ldots, k_l)^t_m)
=\sum_{(a_1, \ldots, a_{r-m}) \in \supp C_m(\bk)}\sum_{i=1}^{r-m}(1, a_{i+1}, \ldots, a_{r-m}, a_1, \ldots, a_i), \nonumber
\end{align}
and the coefficient of 
\begin{align*}
\sum_{l=1}^{r}(k_{l+1}, \ldots, k_r, k_1, \ldots, k_l, 1)^t
-t\sum_{l=1}^{r}(k_{l+1},  \ldots, k_r, k_1, \ldots, k_{l-1}, k_l+1)^t
\end{align*}
at $t^m$ is equal to
\begin{align}\label{eq: 4th of LHS of key}
&\sum_{l=1}^{r}\bigl\{(k_{l+1}, \ldots, k_r, k_1, \ldots, k_l, 1)^t_m
- (k_{l+1}, \ldots, k_r, k_1, \ldots, k_{l-1}, k_l+1)^t_{m-1}\bigr\}\\
&=\sum_{l=1}^{r}((k_{l+1}, \ldots, k_r, k_1, \ldots, k_l)^t_m, 1)
=\sum_{(a_1, \ldots, a_{r-m}) \in \supp C_m(\bk)}\sum_{i=1}^{r-m}(a_{i+1}, \ldots, a_{r-m}, a_1, \ldots, a_i, 1). \nonumber
\end{align}
Therefore, from Lemma \ref{aaaaa}, \eqref{eq: 1st of LHS of key}, \eqref{eq: 2nd of LHS of key}, \eqref{eq: 3rd of LHS of key}, and \eqref{eq: 4th of LHS of key} we have
\begin{align*}
F^t(\bk)_m
=F^0(C_m(\bk)).
\end{align*}
This finishes the proof. 
\end{proof}

\begin{proof}[Proof of Theorem \ref{cycsumF}]
By the results of Kawasaki--Oyama \cite[Theorem~1.2]{KO19} and Hirose--Sato (unpublished) (or Hirose--Murahara--Ono \cite[Theorem~2.4]{HMO20}), we have $\zeta_{\mathcal{F}} (F^0(C_m(\bk)))=0$ for $0 \le m \le r-1$, which leads to
$\zeta_{\mathcal{F}} (F^t(\bk))=0$
because of Proposition~\ref{prop: key of CSF}. 
This completes the proof of Theorem \ref{cycsumF}.
\end{proof}

%%%%%%%%%%%%%%%%%%%%%%%%%%%%%%%%%%%%%%%%%%%%%%%%%%%%%%%%%%%%%%%%%%%%%%%%%%%%%%%%%%%%%%%%%%%%%%%%%%%%%%%
\section{Bowman--Bradley type formula}
Bowman--Bradley \cite[Theorem 5.1]{BB02} proved that the sum of MZVs for indices which consist of the shuffles of two kinds of the strings $\{1,3,\ldots,1,3\}$ and $\{2,\ldots,2\}$ is a rational multiple of a power of $\pi$. 
Saito--Wakabayashi \cite[Theorem 1.4]{SW16} obtained its counterparts for FMZ(S)Vs.
Here, we generalize them to $t$-FMZVs. Note that there seems to be no Bowman--Bradley type theorem in $t$-MZV.

For $(l, m) \in \bZ^2_{\ge0}\setminus\{(0, 0)\}$, $I_{l, m}$ denote the set of all sequences of the form 
\begin{align*}
 \ba=(a_1, \ldots, a_l; b_1, \ldots, b_l; c_1, \ldots, c_m),
\end{align*}
where $a_1, \ldots, a_l$ and $b_1, \ldots, b_l$ are odd positive integers and $c_1, \ldots, c_m$ are even positive integers. 
We define a $\mathbb{Q}$-bilinear product $\sha$ on $\mathcal{I}$ inductively by setting
\begin{align*}
 \boldsymbol{k} \sha \emptyset 
 &=\emptyset \sha \boldsymbol{k}=\boldsymbol{k}, \\ 
 (k _1, \boldsymbol{k}) \sha (l_1,\boldsymbol{l}) 
 &=(k_1, \boldsymbol{k} \sha (l_1,\boldsymbol{l})) + (l_1, (k_1,\boldsymbol{k}) \sha \boldsymbol{l})
\end{align*} 
for all indices $\boldsymbol{k}$, $\boldsymbol{l}$ and all positive integers $k_1$, $l_1$. 
For $\ba=(a_1, \ldots, a_l; b_1, \ldots, b_l; c_1, \ldots, c_m) \in I_{l, m}$, set
\begin{align*}
 B_{\ba}
 &\coloneqq \sum_{\sigma, \tau \in \fS_l}
  (a_{\sigma(1)}, b_{\tau(1)}, \ldots, a_{\sigma(l)}, b_{\tau(l)}) 
  \sha (c_1) \sha \cdots\sha (c_m)  \in \mathcal{I}, 
\end{align*}
where $\fS_r$ is the symmetric group of degree $r$. 
Then, Theorem \ref{BBtype} can be generalized as follows: 
\begin{thm}\label{BBtype_main}
 For any $(l, m) \in \bZ^2_{\ge0} \setminus\{(0, 0)\}$ and $\ba \in I_{l, m}$, we have 
 $\zeta_{\cF}^{t}(B_{\ba})=0$.
\end{thm}
\begin{proof}[Proof of Theorem \ref{BBtype}]
 Theorem \ref{BBtype} is obtained by setting 
 $\ba=(a, \ldots, a; b, \ldots, b; c, \ldots, c) \in I_{l, m}$
 in Theorem \ref{BBtype_main}. 
\end{proof}

For non-negative integers $l, m, n$ with $(l, m)\ne(0, 0)$ 
and $\ba \in I_{l, m}$, 
we denote by $B^{(n)}_{l, m}[\ba]$ the element in $\cI$ which is the coefficient of the $t$-index $B^t_{\ba}$ at $t^n$. Thus we have $B^t_{\ba}=\sum_{n=0}^{2l+m-1}B^{(n)}_{l, m}[\ba]t^n$.
Then, 
we see that the statement of Theorem \ref{BBtype_main} is equivalent to 
\begin{align}\label{BB equiv}
 \zeta_{\cF}\bigl(B^{(n)}_{l, m}[\ba]\bigr)=0
\end{align}
for all $n$ with $0\le  n \le 2l+m-1$. To prove \eqref{BB equiv}, we need the following proposition which relates $B^{(n+1)}_{l, m}[\ba]$ and $B^{(n)}_{l, m}[\ba]$.

\begin{prop}[Key proposition]\label{Key prop for BB}
 For non-negative integers $l, m, n$ with $(l, m)\ne(0, 0)$ 
 and $\ba=(a_1, \ldots, a_l; b_1, \ldots, b_l; c_1, \ldots, c_m) \in I_{l, m}$, we have
 \begin{align}\label{n+1 to n}
  &(n+1)B^{(n+1)}_{l, m}[\ba]\\
  &= 2\sum_{j=1}^m\sum_{i=1}^l\bigl\{B^{(n)}_{l, m-1}[(a_1, \ldots, a_{i-1}, a_i+c_j, a_{i+1}, \ldots, a_l; b_1, \ldots, b_l; c_1, \ldots, c_{j-1}, c_{j+1}, \ldots, c_m)] \nonumber \\
 &\qquad\qquad\quad +B^{(n)}_{l, m-1}[(a_1, \ldots, a_l; b_1, \ldots, b_{i-1}, b_i+c_j, b_{i+1}, \ldots, b_l; c_1, \ldots, c_{j-1}, c_{j+1}, \ldots, c_m)]\bigr\} \nonumber \\
 &\quad +\sum_{i, j=1}^l B^{(n)}_{l-1, m+1}[(a_1, \ldots, a_{i-1}, a_{i+1}, \ldots, a_l; b_1, \ldots, b_{j-1}, b_{j+1} \ldots, b_l; a_i+b_j, c_1, \ldots, c_m)] \nonumber \\
 &\quad +2\sum_{1\le i<j \le m}B^{(n)}_{l, m-1}[(a_1, \ldots, a_l; b_1, \ldots, b_l; c_i+c_j, c_1, \ldots, \overset{\text{remove}}{c_i}, \ldots, \overset{\text{remove}}{c_j}, \ldots, c_m)]. \nonumber
 \end{align}
\end{prop}
\begin{proof}
 We note that the depths of all the indices appearing in \eqref{n+1 to n} are all $s\coloneqq 2l+m-(n+1)$. 
 
 First, consider an index $\bm=(m_1, \ldots, m_s)$ with weight $\sum_{i=1}^la_i+\sum_{i=1}^lb_i+\sum_{j=1}^mc_j$. We can observe that if the index $\bm$ appears in $B^{(n+1)}_{l, m}[\ba]$, each $m_i$ can be written as
 \begin{align} \label{m_i}
  m_i=a^{(i)}_{1}+\cdots+a^{(i)}_{p_i}+b^{(i)}_{1}+\cdots+b^{(i)}_{q_i}+c^{(i)}_{1}+\cdots+c^{(i)}_{r_i}
  \qquad (i=1,\ldots,s),
 \end{align}
 where
 \begin{equation}\label{eq: compositions}
\begin{split}
\{a^{(1)}_1, \ldots, a^{(1)}_{p_1}, \ldots, a^{(s)}_1, \ldots, a^{(s)}_{p_s}\}
  &\coloneqq \{a_1, \ldots, a_l\}, \\
  \{b^{(1)}_1, \ldots, b^{(1)}_{q_1}, \ldots, b^{(s)}_1, \ldots, b^{(s)}_{q_s}\}
   &\coloneqq \{b_1, \ldots, b_l\}, \\
   \{c^{(1)}_1, \ldots, c^{(1)}_{r_1}, \ldots, c^{(s)}_1, \ldots, c^{(s)}_{r_s}\}
   &\coloneqq \{c_1, \ldots, c_m\}
\end{split}
\end{equation}
satisfying
 \begin{equation}\label{eq: sum of p,q,r}
 \begin{split}
 &|p_i-q_i| \le1,\\
  &\sum_{i=1}^s(p_i+q_i+r_i-1)
  =n+1.
  \end{split}
 \end{equation}
Let $\pi$ denote such composition \eqref{eq: compositions} of $\{a_1, \ldots, a_l\}$, $\{b_1, \ldots, b_l\}$ and $\{c_1, \ldots, c_m\}$ satisfying \eqref{eq: sum of p,q,r}. For such composition $\pi$, we set the index $\bm^{\pi} \coloneqq (m^{\pi}_1, \ldots, m^{\pi}_s)$ with
 \begin{align*} %\label{m_i}
  m^{\pi}_i=a^{(i)}_{1}+\cdots+a^{(i)}_{p_i}+b^{(i)}_{1}+\cdots+b^{(i)}_{q_i}+c^{(i)}_{1}+\cdots+c^{(i)}_{r_i}
  \qquad (i=1,\ldots,s).
 \end{align*}
To prove the equality \eqref{n+1 to n}, for any fixed composition $\pi$ satisfying \eqref{eq: sum of p,q,r}, it suffices to show that the number of $\bm^{\pi}$ appearing in the both sides of \eqref{n+1 to n} coincide. 

Let $L$ and $R_1, R_2, R_3$ be the numbers of $\bm^{\pi}$ appearing 
in the left-hand side and the first, the second, the third sums on the right-hand side, respectively.
We calculate $L$, first.  For a fixed composition $\pi$ with \eqref{eq: sum of p,q,r}, 
we find that there are $p_i!q_i!(p_i+q_i+1)_{r_i}$ ways to get each $m^{\pi}_i$. Here,
 \begin{align*}
  (x)_n\coloneqq 
  \begin{cases}
   1 & \text{if $n=0$},\\
   x(x+1)\cdots (x+n-1) & \text{if $n>0$}. 
  \end{cases}
 \end{align*}
 Hence we find
 \begin{align*}
  L=(n+1)\prod_{i=1}^sp_i!q_i!(p_i+q_i+1)_{r_i}.
 \end{align*}

 Put $A_i\coloneqq p_i!q_i!(p_i+q_i+1)_{r_i}$.
 Similarly on the right-hand side, we have
 \begin{align*} 
  R_1
  &=2\sum_{i=1}^s (p_i+q_i)r_ip_i!q_i!(p_i+q_i+1)_{r_i-1}\prod_{j \ne i}A_j, \\
  R_2
  &=\sum_{i=1}^s p_iq_i(p_i-1)!(q_i-1)!(p_i-1+q_i-1+1)_{r_i+1}\prod_{j \ne i}A_j \\
  &=\sum_{i=1}^s p_i!q_i!(p_i+q_i-1)_{r_i+1}\prod_{j \ne i}A_j, \\
  R_3
  &=2\sum_{i=1}^s \binom{r_i}{2}p_i!q_i!(p_i+q_i+1)_{r_i-1}\prod_{j \ne i}A_j.  
 \end{align*}
 Thus we get
 \begin{align*}
  &R_1+R_2+R_3 \\
  &=\sum_{i=1}^s \left( 2(p_i+q_i)r_i+(p_i+q_i-1)(p_i+q_i)+r_i(r_i-1) \right) p_i!q_i!(p_i+q_i+1)_{r_i-1}\prod_{j\ne i}A_j\\
  &=\sum_{i=1}^s(p_i+q_i+r_i-1)p_i!q_i!(p_i+q_i+1)_{r_i}\prod_{j\ne i}A_j\\
  &=\left( \sum_{i=1}^s(p_i+q_i+r_i-1) \right) \prod_{i=1}^sA_i\\
  &=(n+1)\prod_{i=1}^sA_i
  =(n+1)\prod_{i=1}^sp_i!q_i!(p_i+q_i+1)_{r_i}.
 \end{align*}
 This completes the proof.
\end{proof}

\begin{proof}[Proof of Theorem \ref{BBtype_main}]
We prove the statement \eqref{BB equiv} by the induction on $n \ge0$. 
The case $n=0$ is just the usual Bowman--Bradley type formula for FMZVs, which was proved by Saito--Wakabayashi \cite[Theorem 1.4]{SW16}. 

We assume \eqref{BB equiv} holds for $n\ge 0$. 
Then, since $a_i+c_j, b_i+c_j$ are odd for all $i, j$ and $a_i+b_j, c_i+c_j$ are even for all $i, j$, we can use the induction hypothesis for each summand of the right-hand side of \eqref{n+1 to n}. 
Therefore we obtain $\zeta_{\cF}\bigl(B^{(n+1)}_{l, m}[\ba]\bigr)=0$.
\end{proof}

%%%%%%%%%%%%%%%%%%%%%%%%%%%%%%%%%%%%%%%%%%%%%%%%%%%%%%%%%%%%%%%%%%%%%%%%%%%%%%%%%%%%%%%%%%%%%%%%%%%%%%%
\section{Weighted sum formula}
In this section, we prove a weighted sum formula for $t$-FMZVs. 
Many types of weighted sum formulas for MZ(S)Vs are already known (for example, Ohno--Zudilin \cite[Theorem 3]{OZ08}, Machide \cite[Corollary 1.2]{Mac15}, Nakamura \cite[Theorem 1.1, 1.2]{Nak08}, Li--Qin \cite[(1.1), (1.2)]{LQ19}). 
Li--Qin \cite[(3.11)]{LQ16} simplified Guo-Xie's result \cite[Theorem 1.1, 2.2]{GX09} and gave an analogous result \cite[(3,12)]{LQ16} for MZSVs. 
Moreover, Li \cite[Theorem 3.11]{Li19} interpolated Guo--Xie's result and Li--Qin's result for $t$-MZVs. 

On the other hand, another type of the weighted sum formula for FMZ(S)Vs was established recently (see Hirose--Murahara--Saito \cite[Theorem 1.2]{HMS19} for $\cA$-MZVs and Murahara \cite[Theorem 1.1]{Mur18} for $\cS$-MZVs). 
We interpolate the formula in Murahara \cite[Theorem 1.1]{Mur18}, more precisely that of the case $i=r$ and with a slightly different coefficient, for $t$-FMZVs (Theorem \ref{wt_sum_formula}). Note that the counterpart in $t$-MZV of Theorem 1.5 does not seem to exist so far.

For positive integers $k, r$  with $1 \le r \le k$ and a non-negative integer $n$ with $0\le n \le r-1$, we set 
\begin{align*}
 F^t(k,r)
 &\coloneqq \sum_{\substack{k_{1}+\cdots+k_{r}=k\\k_1,\dots,k_r\ge1}}
  2^{k_r-1} \cdot (k_1,\ldots,k_r)^t \in \cI[t], \\
 F(k,r,n)
 &\coloneqq \text{the coefficient of $t^n$ in $F^t(k,r)$} \in \cI.    
\end{align*}

In these notation, we rewrite Theorem \ref{wt_sum_formula}.
\begin{thm} \label{main}
 For a positive integer $k$ and an odd positive integer $r$ with $1\le r\le k$, we have
 \begin{align*}
  \zeta_{\mathcal{F}}^t (F^t(k,r))
  =0.
 \end{align*}
\end{thm}
\begin{rem}
 For the proof of Theorem \ref{main}, it suffices to show $\zeta_{\mathcal{F}} (F(k,r,n))=0$ for all $n$ with $0 \le n\le r-1$. 
 Note that $\zeta_{\cF}(F(k, r, 0))=0$ is essentially equivalent to the formula \cite[Theorem 1.1]{Mur18} with $i=r$.
\end{rem}

To prove Theorem \ref{main}, we prepare some notation. 
\begin{defn}
\begin{enumerate}
\item For positive integers $k,r$ with $1\le r\le k$, we set 
\begin{align*}
 S(k,r)
 &\coloneqq \sum_{\substack{k_{1}+\cdots+k_{r}=k\\k_1,\dots,k_r\ge1}}
  (k_1,\ldots,k_r) \in \cI.     
\end{align*}
\item We always assume that $\be$ runs over sequences of non-negative integers. 
The weight and depth for $\be$ are also defined.
For an index $\boldsymbol{k}$ and a non-negative integer $m$, we set  
\begin{align*}
 G_{1}(\boldsymbol{k},m)
 &\coloneqq \sum_{\substack{\wt(\boldsymbol{e})=m \\ \dep(\boldsymbol{e})=\dep(\boldsymbol{k})}}
  (\boldsymbol{k}\oplus\boldsymbol{e}) \in \cI,\\
 G_{2}(\boldsymbol{k},m)
 &\coloneqq \sum_{\substack{\wt(\boldsymbol{e})=m \\ \dep (\boldsymbol{e})=\dep(\boldsymbol{k}^\vee)}}
  ((\boldsymbol{k}^{\vee}\oplus\boldsymbol{e})^{\vee}) \in \cI, \\
 G(\boldsymbol{k},m)
 &\coloneqq G_{1}(\boldsymbol{k},m)-G_{2}(\boldsymbol{k},m) \in \cI.
\end{align*}
Here, for an index $\boldsymbol{k}=(k_{1}\ldots,k_{r})$ and a sequence of non-negative integers $\boldsymbol{e}=(e_{1},\ldots, e_{r})$ of the same depths,
the symbol $\boldsymbol{k}\oplus\boldsymbol{e}$ represents the componentwise sum, i.e., 
$\boldsymbol{k}\oplus\boldsymbol{e}\coloneqq (k_{1}+e_{1},\ldots,k_{r}+e_{r})$, and $\bk^\vee$ is the Hoffman dual of $\bk$ defined by 
\begin{align*}
 \bk^\vee
 \coloneqq (\underbrace{1, \ldots, 1}_{k_1}+\underbrace{1, \ldots, 1}_{k_2}+1, \ldots, 1+\underbrace{1, \ldots, 1}_{k_r}).
\end{align*}
We extend the Hoffman dual on indices to the $\bQ$-linear map on $\cI$ and we also denote the Hoffman dual of $w \in \cI$ by $w^{\vee}$.
\item For positive integers $k,r$ with $1\le r\le k$ and a non-negative integer $n$ with $n\le r-1$, we set
\begin{align*}
 S'(k, r, n)&\coloneqq -\binom{k-r+n}{n}S(k, r-n), \\
 G'(k, r, n)&\coloneqq -\sum_{m=0}^{k-r-1}2^{k-r-m-1}\binom{m+n}{n}G\bigl((\{1\}^{r-n-1}, k-r-m+1), m+n\bigr), \\
 G'_1(k, r, n)&\coloneqq -\sum_{m=0}^{k-r-1}2^{k-r-m-1}\binom{m+n}{n}G_1\bigl((\{1\}^{r-n-1}, k-r-m+1), m+n\bigr), \\
 G'_2(k, r, n)&\coloneqq -\sum_{m=0}^{k-r-1}2^{k-r-m-1}\binom{m+n}{n}G_2\bigl((\{1\}^{r-n-1}, k-r-m+1), m+n\bigr), \\
 H(k, r, n)&\coloneqq F(k,r,n)+S'(k, r, n)+G'(k, r, n).
\end{align*}
Note that $G'(k, r, n)=G'_1(k, r, n)-G'_2(k, r, n)$. 
\end{enumerate}
\end{defn}

We define a $\bQ$-linear isomorphism $\phi$ on $\cI$ by
\begin{align*} 
 \phi(\bk)
 \coloneqq (-1)^{\dep(\bk)} \sum_{\substack{\square\textrm{ is either a comma `,' } \\
 \textrm{ or a plus `+'}}}(\underbrace{1\square1\square\cdots\square1}_{\substack{\textrm{the number of}\\ \textrm{`1' is }k_1}}, \ldots, \underbrace{1\square1\square\cdots\square1}_{\substack{\textrm{the number of}\\ \textrm{`1' is }k_r}}) \in \cI.
\end{align*}
To prove Theorem \ref{main}, the following proposition plays a key role. 
\begin{prop}[Key proposition] \label{key_lem}
 For a positive integer $k$, an odd positive integer $r$ with $1\le r\le k$, and a non-negative integer $n$ with $n\le r-1$, we have
 \begin{align*}
  &H(k,r,n)+\phi(H(k,r,n)) 
  =
  \begin{cases}
   \displaystyle{
   -\binom{k-r+n}{n} \cdot (\{1\}^{k}) 
   }
   \quad & \textrm{if $k$ is even}, \\
   \quad0 & \textrm{if $k$ is odd}.
  \end{cases}
 \end{align*}
\end{prop}
To prove Proposition \ref{key_lem}, we need Lemmas \ref{F}, \ref{F+S+G_1}, and \ref{G_2}. 
\begin{lem}\label{F}
For positive integers $k$ and $r$ with $1\le r\le k$, we have
\begin{align}\label{F(k,r,n)}
 &F(k,r,n) \\
 &=\sum_{\substack{a_{1}+\cdots+a_{r-n}=k\\a_{1},\dots,a_{r-n}\ge1}}
  \biggl\{
   \sum_{i=1}^{a_{r-n}-1} 2^{i-1} \binom{k-r+n-i}{n-1} +2^{a_{r-n}-1} \binom{k-r+n-a_{r-n}+1}{n}
  \biggr\} \nonumber \\
  &\qquad\qquad\qquad \times (a_1,\ldots,a_{r-n}). \nonumber
\end{align}
\end{lem}
\begin{proof}
We note that the depths of all the indices appearing in \eqref{F(k,r,n)} is $r-n$, and let $d$ be $r-n$.
Fix an index $(a_1, \ldots, a_d)$ appearing in $F(k, r, n)$.
 
Choose $i\; (1\le i \le a_d)$. For $i$, we can count the number of the indices $(a_1, \ldots, a_d)$ appearing in $F(k, r, n)$ by the following steps.
\begin{enumerate}
\item Draw $k$ circles. 
\item Draw lines to divide $k$ circles into $d$ parts corresponding to $(a_1, \ldots, a_d)$. 
\item Draw a dashed line between the $i$-th and the $(i+1)$-th circles from the right.
\item If $1\le i <a_d$, proceed (a). If $i=a_d$, proceed (a').

\begin{description}
\item[(a)]  Draw $(k-1)-i-(d-1)=k-i-d$ arrows (dashed or non-dashed arrows) between the circles without lines and choose $(r-2)-(d-1)=r-d-1$ arrows (dashed arrows) from them. 
\item[(b)] Obtain the index $(k_1, \ldots, k_r)$ with $k_r=i$ by counting the number of the circles between the adjacent two of the (dashed and non-dashed) lines and dashed arrows. 
\item[(c)] The number of indices $(k_1, \ldots, k_r)$ with $k_r=i$ obtained by the above (i) to (iv), (a), and (b) is $\binom{k-r+n-i}{n-1}$. We end all steps in the case of $1\le i <a_d$.

\item[(a')]  Draw $k-i-(d-1)=k-i-d+1$ arrows (dashed or non-dashed arrows) between the circles without lines and choose $(r-1)-(d-1)=r-d$ arrows (dashed arrows) from them. Note that in this case, the rightmost line and the dashed line corresponding to $i$ coincide. 
\item[(b')] Obtain the index $(k_1, \ldots, k_r)$ with $k_r=i$ by counting the number of the circles between the adjacent two of the (dashed and non-dashed) lines and dashed arrows.
\item[(c')] The number of indices $(k_1, \ldots, k_r)$ with $k_r=i$ obtained by the above (i) to (iv), (a'), and (b') is $\binom{k-r+n-i+1}{n}$. We end all steps in the case of $i=a_d$.
\end{description}

\end{enumerate}

For example, consider the number of $(5, 3, 4, 3)$ appearing in $F^t(15, 7)$ and $i=2$. From the above steps (i) to (iv) and (a) to (c), we obtain the following figure.
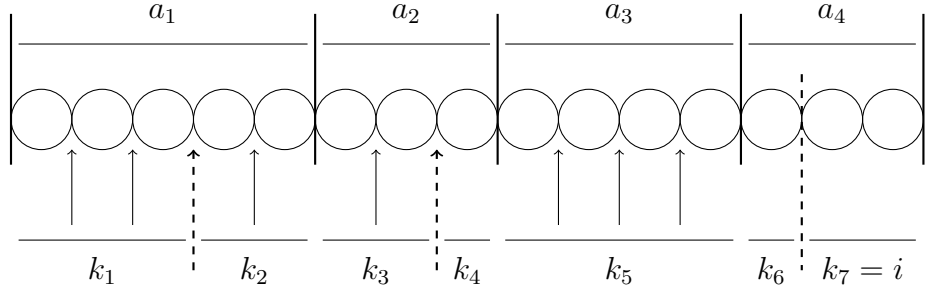
\begin{align*}
\begin{tikzpicture}
\draw(0,0) circle (0.4cm);
\draw(0.8,0) circle (0.4cm);
\draw(1.6,0) circle (0.4cm);
\draw(2.4,0) circle (0.4cm);
\draw(3.2,0) circle (0.4cm);
\draw(4.0,0) circle (0.4cm);
\draw(4.8,0) circle (0.4cm);
\draw(5.6,0) circle (0.4cm);
\draw(6.4,0) circle (0.4cm);
\draw(7.2,0) circle (0.4cm);
\draw(8.0,0) circle (0.4cm);
\draw(8.8,0) circle (0.4cm);
\draw(9.6,0) circle (0.4cm);
\draw(10.4,0) circle (0.4cm);
\draw(11.2,0) circle (0.4cm);
\draw[thick](-0.4,1.4)--(-0.4,-0.6);
\draw[thick](3.6,1.4)--(3.6,-0.6);
\draw[thick](6.0,1.4)--(6.0,-0.6);
\draw[thick](9.2,1.4)--(9.2,-0.6);
\draw[thick](11.6,1.4)--(11.6,-0.6);
\draw[->](0.4,-1.4)--(0.4,-0.4);
\draw[->](1.2,-1.4)--(1.2,-0.4);
\draw[->](2.8,-1.4)--(2.8,-0.4);
\draw[->](4.4,-1.4)--(4.4,-0.4);
\draw[->](6.8,-1.4)--(6.8,-0.4);
\draw[->](7.6,-1.4)--(7.6,-0.4);
\draw[->](8.4,-1.4)--(8.4,-0.4);
\draw[thick,dashed][->](2.0,-2.0)--(2.0,-0.4);
\draw[thick,dashed][->](5.2,-2.0)--(5.2,-0.4);
\draw[thick,dashed](10.0,0.6)--(10.0,-2.0);
\draw(-0.3,1.0)--(3.5,1.0); 
\draw(1.6,1.4) node{$a_1$}; 
\draw(3.7,1.0)--(5.9,1.0); 
\draw(4.8, 1.4) node{$a_2$};  
\draw(6.1,1.0)--(9.1,1.0); 
\draw(7.6, 1.4) node{$a_3$}; 
\draw(9.3,1.0)--(11.5,1.0); 
\draw(10.4, 1.4) node{$a_4$};  
\draw(-0.3,-1.6)--(1.9,-1.6); 
\draw(0.8,-2.0) node{$k_1$}; 
\draw(2.1,-1.6)--(3.5,-1.6); 
\draw(2.8,-2.0) node{$k_2$}; 
\draw(3.7,-1.6)--(5.1,-1.6); 
\draw(4.4,-2.0) node{$k_3$}; 
\draw(5.3,-1.6)--(5.9,-1.6); 
\draw(5.6,-2.0) node{$k_4$}; 
\draw(6.1,-1.6)--(9.1,-1.6); 
\draw(7.6,-2.0) node{$k_5$}; 
\draw(9.3,-1.6)--(9.9,-1.6); 
\draw(9.6,-2.0) node{$k_6$}; 
\draw(10.1,-1.6)--(11.5,-1.6); 
\draw(10.8,-2.0) node{$k_7=i$}; 
\end{tikzpicture}
\end{align*}

Thus we see that the number of $(a_1, \ldots, a_d)$ appearing in $F^t(k, r)$ coincides with
\begin{align}\label{first count}
\sum_{i=1}^{a_{r-n}-1}2^{i-1}\binom{k-r+n-i}{n-1}. 
\end{align}

On the other hand, consider the number of $(5, 3, 4, 3)$ appearing in $F^t(15, 7)$ and $i=3$
 From the above steps (i) to (iv) and (a') to (c'), we obtain the following figure.
 
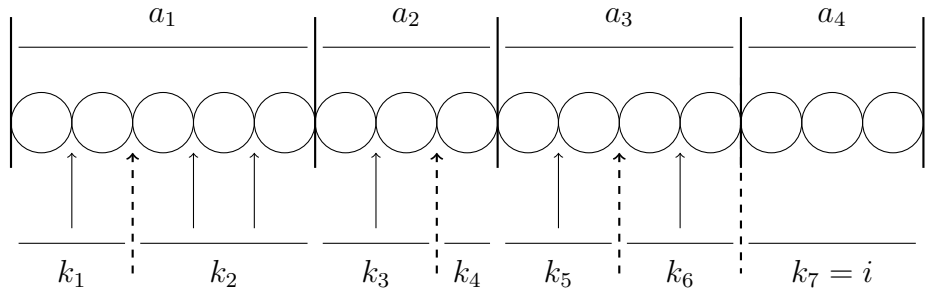
\begin{align*}
\begin{tikzpicture}
\draw(0,0) circle (0.4cm);
\draw(0.8,0) circle (0.4cm);
\draw(1.6,0) circle (0.4cm);
\draw(2.4,0) circle (0.4cm);
\draw(3.2,0) circle (0.4cm);
\draw(4.0,0) circle (0.4cm);
\draw(4.8,0) circle (0.4cm);
\draw(5.6,0) circle (0.4cm);
\draw(6.4,0) circle (0.4cm);
\draw(7.2,0) circle (0.4cm);
\draw(8.0,0) circle (0.4cm);
\draw(8.8,0) circle (0.4cm);
\draw(9.6,0) circle (0.4cm);
\draw(10.4,0) circle (0.4cm);
\draw(11.2,0) circle (0.4cm);
\draw[thick](-0.4,1.4)--(-0.4,-0.6);
\draw[thick](3.6,1.4)--(3.6,-0.6);
\draw[thick](6.0,1.4)--(6.0,-0.6);
\draw[thick](9.2,1.4)--(9.2,-0.6);
\draw[thick](11.6,1.4)--(11.6,-0.6);
\draw[->](0.4,-1.4)--(0.4,-0.4);
\draw[->](2.0,-1.4)--(2.0,-0.4);
\draw[->](2.8,-1.4)--(2.8,-0.4);
\draw[->](4.4,-1.4)--(4.4,-0.4);
\draw[->](6.8,-1.4)--(6.8,-0.4);
\draw[->](8.4,-1.4)--(8.4,-0.4);
\draw[thick,dashed][->](1.2,-2.0)--(1.2,-0.4);
\draw[thick,dashed][->](5.2,-2.0)--(5.2,-0.4);
\draw[thick,dashed][->](7.6,-2.0)--(7.6,-0.4);
\draw[thick,dashed](9.2,0.6)--(9.2,-2.0);
\draw(-0.3,1.0)--(3.5,1.0); 
\draw(1.6,1.4) node{$a_1$}; 
\draw(3.7,1.0)--(5.9,1.0); 
\draw(4.8, 1.4) node{$a_2$}; 
\draw(6.1,1.0)--(9.1,1.0);
\draw(7.6, 1.4) node{$a_3$};  
\draw(9.3,1.0)--(11.5,1.0); 
\draw(10.4, 1.4) node{$a_4$};
\draw(-0.3,-1.6)--(1.1,-1.6); 
\draw(0.4,-2.0) node{$k_1$}; 
\draw(1.3,-1.6)--(3.5,-1.6); 
\draw(2.4,-2.0) node{$k_2$}; 
\draw(3.7,-1.6)--(5.1,-1.6); 
\draw(4.4,-2.0) node{$k_3$}; 
\draw(5.3,-1.6)--(5.9,-1.6); 
\draw(5.6,-2.0) node{$k_4$}; 
\draw(6.1,-1.6)--(7.5,-1.6); 
\draw(6.8,-2.0) node{$k_5$}; 
\draw(7.7,-1.6)--(9.1,-1.6); 
\draw(8.4,-2.0) node{$k_6$}; 
\draw(9.3,-1.6)--(11.5,-1.6); 
\draw(10.4,-2.0) node{$k_7=i$}; 
\end{tikzpicture}
\end{align*}
Thus we see that the number of $(a_1, \ldots, a_{r-n})$ with $i=a_{r-n}$ in $F^t(k, r)$ coincides with 
\begin{align}\label{second count}
2^{a_{r-n}-1}\binom{k-r+n-a_{r-n}+1}{n}.
\end{align}
From \eqref{first count} and \eqref{second count}, we obtain the desired formula.
\end{proof}
\begin{lem}\label{F+S+G_1}
 For positive integer $k$, a positive integer $r$ with $1\le r\le k$, and a non-negative integer $n$ with $n\le r-1$, we have
 \begin{align*}
  F(k, r, n)+S'(k, r, n)+G'_1(k, r, n)=0.
 \end{align*}
\end{lem}
\begin{proof}
It is easy to see that
\begin{align}\label{eq:calc of S'}
 S'(k,r,n)
 &=-\binom{k-r+n}{n} 
  \sum_{\substack{a_{1}+\cdots+a_{r-n}=k\\a_1,\dots,a_{r-n}\ge1}}
  (a_1,\ldots,a_{r-n}) \\
 &=-\binom{k-r+n}{n} 
  \sum_{a_{r-n}=1}^{k-r+n+1}
  \sum_{\substack{a_{1}+\cdots+a_{r-n-1}=k-a_{r-n}\\a_1,\dots,a_{r-n-1}\ge1}}
  (a_1,\ldots,a_{r-n}). \nonumber
\end{align}
 By Lemma \ref{F} and \eqref{eq:calc of S'}, 
 we have
\begin{align*}
 &F(k,r,n)+S'(k,r,n) \\
 &=\sum_{a_{r-n}=1}^{k-r+n+1}
  \sum_{\substack{a_{1}+\cdots+a_{r-n-1}=k-a_{r-n}\\a_{1},\dots,a_{r-n-1}\ge1}} \\ 
  &\quad 
  \biggl(
   \sum_{i=1}^{a_{r-n}-1} 2^{i-1} \binom{k-r+n-i}{n-1}
   +2^{a_{r-n}-1} \binom{k-r+n-a_{r-n}+1}{n}
   -\binom{k-r+n}{n} 
  \biggr) \\ 
  &\quad \times (a_1,\ldots,a_{r-n})\\ 
  &=\sum_{a_{r-n}=2}^{k-r+n+1}
  \sum_{\substack{a_{1}+\cdots+a_{r-n-1}=k-a_{r-n}\\a_{1},\dots,a_{r-n-1}\ge1}} \\ 
  &\quad 
  \biggl(
   \sum_{i=1}^{a_{r-n}-1} 2^{i-1} \binom{k-r+n-i}{n-1}
   +2^{a_{r-n}-1} \binom{k-r+n-a_{r-n}+1}{n}
   -\binom{k-r+n}{n} 
  \biggr) \\ 
  &\quad \times (a_1,\ldots,a_{r-n}). 
\end{align*}
On the other hand, it is easy to see that
\begin{align*}
 G_1((\{1\}^{r-n-1},k-r-m+1),m+n) 
 =\sum_{\substack{a_{1}+\cdots+a_{r-n}=k\\a_{1},\dots,a_{r-n-1}\ge1\\a_{r-n}\ge k-r-m+1}}
  (a_1,\ldots,a_{r-n}).
\end{align*}
Then we have
\begin{align*}
 &G'_1(k,r,n) \\
 &=-\sum_{m=0}^{k-r-1} 2^{k-r-m-1} \binom{m+n}{n} 
  \sum_{\substack{a_{1}+\cdots+a_{r-n}=k\\a_{1},\dots,a_{r-n-1}\ge1\\a_{r-n}\ge k-r-m+1}}
  (a_1,\ldots,a_{r-n}) \\  
 &=-\sum_{m=0}^{k-r-1} 2^{k-r-m-1} \binom{m+n}{n} 
  \sum_{a_{r-n}=k-r-m+1}^{k-r+n+1}
  \sum_{\substack{a_{1}+\cdots+a_{r-n-1}=k-a_{r-n}\\a_{1},\dots,a_{r-n-1}\ge1}}
  (a_1,\ldots,a_{r-n}) \\ 
 &=-\sum_{m=0}^{k-r-1} 2^{m} \binom{k-r-m+n-1}{n} 
  \sum_{a_{r-n}=m+2}^{k-r+n+1}
  \sum_{\substack{a_{1}+\cdots+a_{r-n-1}=k-a_{r-n}\\a_{1},\dots,a_{r-n-1}\ge1}}
  (a_1,\ldots,a_{r-n}) \\ 
 &=\biggl( 
  -\sum_{a_{r-n}=2}^{k-r+n+1} \sum_{m=0}^{a_{r-n}-2} +\sum_{a_{r-n}=k-r+1}^{k-r+n+1} \sum_{m=k-r}^{a_{r-n}-2}
  \biggr) \,
  2^{m} \binom{k-r-m+n-1}{n} \\  
 &\qquad \times \sum_{\substack{a_{1}+\cdots+a_{r-n-1}=k-a_{r-n}\\a_{1},\dots,a_{r-n-1}\ge1}}
  (a_1,\ldots,a_{r-n}).
\end{align*}
Note that by the definition of binomial coefficients, we have
\begin{align*}
\binom{k-r-m+n-1}{n}=0
\end{align*}
for $k-r \le m$. Thus we have
 \begin{align*}
 G'_1(k, r, n)=-\sum_{a_{r-n}=2}^{k-r+n+1}\sum_{m=0}^{a_{r-n}-2}\sum_{\substack{a_{1}+\cdots+a_{r-n-1}=k-a_{r-n}\\a_{1},\dots,a_{r-n-1}\ge1}}
 2^{m} 
 \binom{k-r-m+n-1}{n}
 (a_1,\ldots,a_{r-n}).
 \end{align*}
From this, if we set $a\coloneqq a_{r-n}$, we obtain
 \begin{align}\label{first calc of FSG}
 &F(k, r, n)+S'(k, r, n)+G'_1(k, r, n) \\
 &=\sum_{a=2}^{k-r+n+1}\sum_{\substack{a_1+\cdots+a_{r-n-1}=k-a \\ a_1, \ldots, a_{r-n-1} \ge1}}\left\{\sum_{m=1}^{a-1}2^{m-1}\binom{k-r+n-m}{n-1}+2^{a-1}\binom{k-r+n-a+1}{n}\right.\nonumber \\
 &\quad \left. -\binom{k-r+n}{n}-\sum_{m=0}^{a-2}2^m\binom{k-r-m+n-1}{n} \right\}(a_1, \ldots, a_{r-n-1}, a). \nonumber
 \end{align}
Here we note that
\begin{align}\label{second calc of FSG}
&\sum_{m=1}^{a-1}2^{m-1}\binom{k-r+n-m}{n-1}-\sum_{m=0}^{a-2}2^m\binom{k-r+n-m-1}{n} \\
&=\sum_{m=0}^{a-2}2^m\left\{\binom{k-r+n-m-1}{n-1}-\binom{k-r+n-m-1}{n}\right\} \nonumber\\
&=\sum_{m=0}^{a-2}2^m\left\{\binom{k-r+n-m}{n}-2\binom{k-r+n-m-1}{n}\right\} \nonumber\\
&=\sum_{m=0}^{a-2}\left\{2^m\binom{k-r+n-m}{n}-2^{m+1}\binom{k-r+n-(m+1)}{n}\right\} \nonumber \\
&=\binom{k-r+n}{n}-2^{a-1}\binom{k-r+n-a+1}{n}. \nonumber
\end{align}
From \eqref{first calc of FSG} and \eqref{second calc of FSG}, we finish the proof of this lemma.
\end{proof}
\begin{lem}\label{G_2}
 For a positive integer $k$, an odd positive integer $r$ with $1\le r\le k$, and a non-negative integer $n$ with $n\le r-1$, we have
 \begin{align*} 
  G'_2(k, r, n)+\phi\bigl(G'_2(k, r, n)\bigr)=
  \begin{cases}
   \displaystyle\binom{k-r+n}{n}\cdot\bigl(\{1\}^k\bigr) & \text{if $k$ is even}, \\
   0 & \text{if $k$ is odd}.
  \end{cases}
 \end{align*}
\end{lem}
\begin{proof}
Since
\begin{align*}
 G_2((\{1\}^{r-n-1},k-r-m+1),m+n) 
 &=(G_1((r-n,\{1\}^{k-r-m}),m+n))^\vee \\
 &=\sum_{\substack{b_{1}+\cdots+b_{m+n+1}=k-r+n+1\\b_{1},\dots,b_{m+n+1}\ge1}}
  (\{1\}^{r-n-1},b_1,\ldots,b_{m+n+1}),
\end{align*}
we have
\begin{align*}
G'_2(k, r, n)=-\sum_{m=0}^{k-r-1}2^{k-r-m-1}\binom{m+n}{n}\sum_{\substack{b_1+\cdots+b_{m+n+1}=k-r+n+1 \\ b_1, \ldots, b_{m+n+1}\ge1}}(\{1\}^{r-n-1}, b_1, \ldots, b_{m+n+1}).
\end{align*}
Fix a positive integer $d$ with $1\le d \le k-r-1$ and an index $(b_1, \ldots, b_{d+n+1})$ with weight $k-r+n+1$. We count the number of the index $(\{1\}^{r-n-1}, b_1, \ldots, b_{d+n+1})$ appearing in $G'_2(k, r, n)+\phi(G'_2(k, r, n))$. 

First, it is easy to see that the number of the index $(\{1\}^{r-n-1}, b_1, \ldots, b_{d+n+1})$ appearing in $G'_2(k, r, n)$ is equal to 
 \begin{align*}
 -2^{k-r-d-1}\binom{d+n}{n}
 \end{align*}
since it appears in $G'_2(k, r, n)$ if and only if $m=d$ holds.

Next, we count the number of the index $(\{1\}^{r-n-1}, b_1, \ldots, b_{d+n+1})$ appearing in $\phi(G'_2(k, r, n))$. By the definition of $\phi$, the index $(\{1\}^{r-n-1}, b_1, \ldots, b_{m+n+1})$ in $\phi(G'_2(k,r,n))$ is obtained from the index of the form $(\{1\}^{r-n-1}, b_1\square \cdots \square b_{d+n+1})$ in $G'_2(k,r,n)$ where $(d-m)$ squares are filled with pluses (`+') and the others with commas (`,'), and the number of this way is $\binom{d+n}{d-m}$. Thus, the number of $(\{1\}^{r-n-1}, b_1, \ldots, b_{d+n+1})$ in $\phi(G'_2(k, r, n))$ is equal to 
 \begin{align*}
 -\sum_{m=0}^d(-1)^{r+m}2^{k-r-m-1}\binom{m+n}{n}\binom{d+n}{d-m}.
 \end{align*}

Then we see that the number of such index $(\{1\}^{r-n-1}, b_1, \ldots, b_{m+n+1})$ in $G'_2(k, r, n)+\phi(G'_2(k, r, n))$ coincides with 
\begin{align*}
&-2^{k-r-d-1}\binom{d+n}{n}-\sum_{m=0}^d (-1)^{r+m}2^{k-r-m-1}\binom{m+n}{n}\binom{d+n}{d-m}.
\end{align*}
Note that since 
\begin{align*}
\binom{m+n}{n}\binom{d+n}{d-m}=\binom{d+n}{n}\binom{d}{m}
\end{align*}
and 
\begin{align*}
\sum_{m=0}^d(-2)^{-m}\binom{d}{m}=\left(1-\frac{1}{2}\right)^d=2^{-d}
\end{align*}
by the binomial theorem, we have
\begin{align*}
&-2^{k-r-d-1}\binom{d+n}{n}-\sum_{m=0}^d (-1)^{r+m}2^{k-r-m-1}\binom{m+n}{n}\binom{d+n}{d-m}\\
&=-2^{k-r-d-1}\binom{d+n}{n}-(-1)^r2^{k-r-1}\binom{d+n}{n}\sum_{m=0}^d(-2)^{-m}\binom{d}{m}\\
&=-2^{k-r-d-1}\binom{d+n}{n}\bigl(1+(-1)^r\bigr).
\end{align*}
Since $r$ is odd, the above sum is 0. 
Thus the index $(\{1\}^{r-n-1}, b_1, \ldots, b_{d+n+1})$ of weight $k$ does not appear in $G'_2(k, r, n)+\phi(G'_2(k, r, n))$.

Similarly, we see that the number of the index $(\{1\}^k)$ in $G'_2(k, r, n)+\phi(G'_2(k, r, n))$ coincides with 
\begin{align*}
 &-\sum_{m=0}^{k-r-1} (-1)^{r+m} 2^{k-r-m-1}   
 \binom{m+n}{n} \binom{k-r+n}{k-r-m}\\
 &=-\frac{(-1)^{k+1}+(-1)^r}{2} \binom{k-r+n}{k-r}= \frac{(-1)^k+1}{2}\binom{k-r+n}{n},
\end{align*}
which comes from that $r$ is odd. Thus we obtain the desired formula.
\end{proof}
\begin{proof}[Proof of Proposition \ref{key_lem}]
 By Lemma \ref{F}, Lemma \ref{F+S+G_1} and Lemma \ref{G_2}, we have
 \begin{align*}
   &H(k, r, n)+\phi(H(k, r, n))\\
 &=F(k, r, n)+S'(k, r, n)+G'_1(k, r, n)\\
 &\quad +\phi\bigl(F(k, r, n)+S'(k, r, n)+G'_1(k, r, n)\bigr)+G'_2(k, r, n)+\phi(G'_2(k, r, n))\\
 &=\frac{(-1)^k+1}{2}\binom{k-r+n}{n}\bigl(\{1\}^k\bigr),
 \end{align*}
which is the desired formula.
\end{proof}
\begin{proof}[Proof of Theorem \ref{main}]
 By the duality relation for $\zeta_{\cF}(\bk)$ (\cite[Theorem 4.7]{Hof15}, \cite[Corollarie 1.12]{Jar14}), we have
 \[
  \zeta_{\mathcal{F}} (H(k,r,n)
  =\zeta_{\mathcal{F}} \bigl(\phi(H(k,r,n))\bigr).
 \]
 Then, by Proposition \ref{key_lem} and the fact $\zeta_{\cF}(\{1\}^m)=0$ for a positive integer $m$, we have
 \[
  \zeta_{\mathcal{F}} (H(k,r,n))=0.
 \]
 Moreover, by the fact $\zeta_{\cF}(S(k, r))=0$ (\cite[Theorem 4.1]{Hof15}, \cite[Theorem 1.1]{Mur15}) and the Ohno-type relation $\zeta_{\cF}(G(\bk, m))=0$ (\cite[Theorem 1.5]{Oya18}), we get 
 \[
  \zeta_{\mathcal{F}} (F(k,r,n))=0.
 \]
 This finishes the proof. 
\end{proof}

%%%%%%%%%%%%%%%%%%%%%%%%%%%%%%%%%%%%%%%%%%%%%%%%%%%%%%%%%%%%%%%%%%%%%%%%%%%%%%%%%%%%%%%%%%%%%%%%%%%%%%%%%%
\section{Other relations}
In this section, we introduce several $\bQ$-linear relations among $t$-FMZVs.
First, we give an algebraic setup introduced by Yamamoto \cite{Yam13}, Li \cite{Li19} and Tanaka--Wakabayashi \cite{TW16}, which is based on those for classical MZ(S)Vs (see Hoffman \cite{Hof97} and Muneta \cite{Mun09}).

Let $\mathfrak{H}_t\coloneqq \mathbb{Q}\langle x, y \rangle[t]$ denote the non-commutative polynomial algebra over $\mathbb{Q}[t]$ in two variables $x$ and $y$, and let $\mathfrak{H}^1_t$ and $\mathfrak{H}^0_t$ denote the subalgebras $\mathbb{Q}[t]+y\mathfrak{H}_t$ and $\mathbb{Q}[t]+y\mathfrak{H}_tx$, respectively. 
We write $\mathfrak{H}_0(=\mathbb{Q}\langle x, y \rangle), \mathfrak{H}^1_0, \mathfrak{H}^0_0$ simply by $\mathfrak{H}, \mathfrak{H}^1, \mathfrak{H}^0$, respectively. 
By setting $z_k\coloneqq yx^{k-1}$ for a positive integer $k$, we can identify $\fH^1$ with $\cI$ by the correspondence $z_{k_1}\cdots z_{k_r} \leftrightarrow (k_1, \ldots, k_r)$. 
We denote by $Z_{\cA}$ (resp.\ $Z_{\cS}$) the extension to the $\bQ[t]$-linear map $Z_{\cA} \colon \fH^1_t \rightarrow \cA[t]$ (resp.\ $Z_{\cS} \colon \fH^1_t \rightarrow (\cZ/\zeta(2)\cZ)[t]$).

 %the map $Z_{\cA}$ (resp. $Z_{\cS}$) to the $\bQ[t]$-linear map $\fH^1_t \rightarrow \cA[t]$ (resp. $\fH^1_t \rightarrow (\cZ/\zeta(2)\cZ)$) and we denote it 
 
We define the $\mathbb{Q}[t]$-linear maps $\mathit{Z}^{t}_{\mathcal{A}}\colon\mathfrak{H}^1_t \rightarrow \mathcal{A}[t]$ 
and $\mathit{Z}^{t}_{\mathcal{S}}\colon\mathfrak{H}^1_t \rightarrow (\mathcal{Z}/\zeta(2)\mathcal{Z})[t]$ 
by $\mathit{Z}^{t}_{\mathcal{F}}(1)\coloneqq 1$ 
and
\[
 \mathit{Z}^{t}_{\mathcal{F}}(z_{k_1}\cdots z_{k_r})\coloneqq \zeta^t_{\mathcal{F}}(k_1, \ldots, k_r). 
\]
%

%%%%%%%%%%%%%%%%%%%%%%%%%%%%%%%%%%%%%%%%%%%%%%%%%%%%

The $\mathbb{Q}[t]$-linear map $S_t\colon\mathfrak{H}^1_t \rightarrow \mathfrak{H}^1_t$ is defined by $S_t(1)\coloneqq 1$ and
\[
 S_t(yw)\coloneqq y\sigma_t(w) \quad (w \in \mathfrak{H}_t), 
\]
where $\sigma_t$ is an automorphism on $\mathfrak{H}_t$ characterized by
\[
 \sigma_t(x)\coloneqq x, \quad \sigma_t(y)\coloneqq tx+y.
\]
Then we find that $S_t$ is invertible (in fact $S^{-1}_t=S_{-t}$) and
\[
 \mathit{Z}^{t}_{\cF}=\mathit{Z}_{\cF} \circ S_t.
\]
Moreover, we see that $S_t(\mathfrak{H}^1_t)=\mathfrak{H}^1_t, S_t(\mathfrak{H}^0_t)=\mathfrak{H}^0_t$, 
and $S_t(y\mathfrak{H}_tx)=y\mathfrak{H}_tx$ holds (see Li \cite[p.4]{Li19}).

%%%%%%%%%%%%%%%%%%%%%%%%%%%%%%%%%%%%%%%%%%%%%%%%%%%%

\subsection{Harmonic relation}
The $t$-harmonic product $\overset{t}{\ast}$ introduced by Yamamoto \cite[Definition 3.7]{Yam13} 
is a $\bQ[t]$-bilinear map $\overset{t}{\ast}\colon\fH^1_t \times \fH^1_t \rightarrow \fH^1_t$ defined by the rules
\[
 \begin{cases}
 &1\hart w=w\hart1=w, \\
 &z_kw_1\hart z_lw_2=z_k(w_1\hart z_lw_2)+z_l(z_kw_1\hart w_2)+(1-2t)z_{k+l}(w_1\hart w_2)\\
 &\qquad\qquad\qquad +(1-\delta(w_1)\delta(w_2))(t^2-t)x^{k+l}(w_1\hart w_2),
 \end{cases}
\]
where $w, w_1, w_2 \in \{x, y\}^* \cap \fH^1_t$ and $k$ and $l$ are positive integers . 
Here, $\{x, y\}^*$ is the set of all words generated by the letters $x$ and $y$, and 
the map $\delta\colon\{x, y\}^* \rightarrow \{0, 1\}$ is defined by 
 \[
 \delta(w)\coloneqq 
 \begin{cases}
 1 & \text{if $w=1$}, \\
 0 & \text{if $w\ne1$}.
 \end{cases}
 \] 
Then it is obvious that $\overset{0}{\ast}$ coincides with the usual harmonic product $*$ on $\fH^1$ (see Hoffman \cite{Hof97}) 
and $\overset{1}{\ast}$ coincides with Muneta's $n$-harmonic product $\harb$ on $\fH^1$ defined in \cite[p.9]{Mun09}.

We define the $\bQ$-linear map $\mathit{Z}\colon \fH^0 \rightarrow \bR$ by $\mathit{Z}(1)\coloneqq 1$ and $\mathit{Z}(z_{k_1}\cdots z_{k_r})\coloneqq \zeta(k_1, \ldots, k_r)$.
Then, the harmonic relation for classical MZVs, i.e.,
\[
 \mathit{Z}(w_1 \ast w_2)=\mathit{Z}(w_1)\mathit{Z}(w_2) 
\]
is naturally derived from the series representation. 
Similarly for FMZVs, the following relation holds.
\begin{prop} \label{har rel}
 For $w_1, w_2 \in \fH^1$, we have
 \[
  \mathit{Z}_{\cF}(w_1 \ast w_2)=\mathit{Z}_{\cF}(w_1)\mathit{Z}_{\cF}(w_2).
 \]
\end{prop}
In a similar way for $t$-MZVs, we have the following. 
\begin{thm}[Harmonic relation]\label{harF}
For $w_1, w_2 \in \fH^1_t$, we have
 \[
 \mathit{Z}^{t}_{\cF}(w_1 \overset{t}{\ast} w_2)=\mathit{Z}^{t}_{\cF}(w_1)\mathit{Z}^{t}_{\cF}(w_2).
 \]
\end{thm}
The symmetric sum formula for FMZVs was proved by Hoffman \cite[Theorem 4.5]{Hof15} for $\cA$-MZ(S)Vs and by the first named author \cite{Mur15} for $\cS$-MZ(S)Vs. 
We give a generalization of these theorems for $t$-FMZVs.
\begin{thm}[Symmetric sum formula] \label{symsum}
 For a non-empty index $(k_1,\dots,k_r)$, we have
 \begin{align*}
  \sum_{\sigma \in \fS_r} \zeta_\mathcal{F}^t (k_{\sigma(1)},\dots, k_{\sigma(r)})
  =0.
 \end{align*}
\end{thm}
\begin{rem}
The work by Li \cite[Theorem 3.2]{Li19} is known for the counterpart in $t$-MZVs of Theorem 5.3. 
\end{rem}
Lastly, we give an antipode-like relation for $t$-FMZVs, which is derived from the Hopf algebra structure of the quasi-symmetric functions (see Hoffman \cite{Hof15}). 
\begin{prop}[Antipode-like relation] \label{antipode}
 For a non-empty index $(k_1,\dots,k_r)$, we have
 \begin{align*}
  \sum_{i=0}^{r}(-1)^{i} 
   \zeta_\mathcal{\mathcal{F}}^{t} (k_1,\dots,k_{i})
   \zeta_\mathcal{F}^{1-t} (k_{r,},\dots,k_{i+1})=0.
 \end{align*}
 Here, we understand $\zeta_\mathcal{F}^t (\emptyset)=1$.
\end{prop}
\begin{rem}
 It is well-known that the same relation holds for various MZVs with the $t$-harmonic product structure. 
 For related works, see Hoffman \cite{Hof00}, Zlobin \cite[Theorem 3]{Zlo05}, Kawashima \cite[Proposition 7]{Kaw09}, Ihara--Kajikawa--Ohno--Okuda \cite[Proposition 6]{IKOO11}, Yamamoto \cite[Proposition 3.9]{Yam13}, Hoffman \cite[Theorem 3.1]{Hof15}, Saito \cite[Proposition 2.9]{Sai17}, and Seki \cite[Proposition 5.9]{Sek17}.
\end{rem}

%%%%%%%%%%%%%%%%%%%%%%%%%%%%%%%%%%%%%%
\subsection{Shuffle relation}
The $t$-shuffle product $\sht$, introduced by Li \cite{Li19}, is the $\bQ[t]$-bilinear map $\sht\colon\fH_t \times \fH_t \rightarrow \fH_t$ defined by the rules
\[
 \begin{cases}
 & 1 \sht w=w\sht1=w, \\
 &aw_1 \sht bw_2=a(w_1\sht bw_2)+b(aw_1\sht w_2)-\delta(w_1)\rho(a)bw_2-\delta(w_2)\rho(b)aw_1,
 \end{cases}
\]
where $w, w_1, w_2 \in \{x, y\}^*$ and $a, b \in \{x, y\}$. 
Here, $\rho\colon\{x, y\} \rightarrow \fH_t$ is the map defined by 
\[
 \rho(x)\coloneqq 0, \quad \rho(y)\coloneqq ty.
\]
Then it is obvious that $\overset{0}{\sh}$ coincides with the usual shuffle product $\sh$, but $\overset{1}{\sh}$ is
different from Muneta's $n$-shuffle product in \cite{Mun09}.
We note that the product $\sht$ is associative and commutative as mentioned by Li \cite{Li19}.

The shuffle relation of MZVs, which is naturally derived from the iterated integral representation, states that   
\begin{align} \label{shMZV}
 Z(w_1 \sh w_2)=Z(w_1)Z(w_2) \quad (w_1, w_2 \in \fH^0). 
\end{align}
We define a $\bQ$-linear map $\nu\colon\fH^1 \rightarrow \fH^1$ by $\nu(1)=1$ and $\nu(z_{k_1}\cdots z_{k_r})=(-1)^{k_1+\cdots+k_r}z_{k_r}\cdots z_{k_1}$ for positive integers $k_1, \ldots, k_r$,
and extend this $\bQ[t]$-linearly on $\fH^1_t$. 
The following relation obtained by Kaneko--Zagier \cite{KZ19} is considered to be the counterpart of \eqref{shMZV} for FMZVs. 
\begin{thm}
 For $w_1, w_2 \in \fH^1$, we have
 \begin{align*}
  Z_{\cF}(w_1 \sh w_2)=Z_{\cF}\bigl(w_1 \nu(w_2)\bigr).
 \end{align*}
 Especially when $w_1=1$ and $w_2=w$, we have
 \[
  \mathit{Z}_{\cF}(w)=\mathit{Z}_{\cF}(\nu(w)). 
 \]
\end{thm}
As a generalization of this theorem, we introduce the following theorem for $t$-FMZVs.
To state the theorem, let $L_y : \fH \rightarrow \fH$ be the $\bQ$-linear map defined by $L_y(w)\coloneqq yw$ for $w \in \fH$. We extend it to the $\bQ[t]$-linear map $L_y : \fH_t \rightarrow \fH_t$. Note that the inverse map $L^{-1}_y : y\fH_t \rightarrow \fH_t$ is defined by $L^{-1}_y(yw)\coloneqq w$ for $w \in \fH_t$.
\begin{thm}[Shuffle relation]\label{shF}
 For $w_1, w_2 \in y\fH_t$, we have
 \begin{align*}
  \mathit{Z}^{t}_{\cF}(w_1 \overset{t}{\sh} w_2)
  =\mathit{Z}^{t}_{\cF}\bigl(w_1\nu(w_2)-w_1xL^{-1}_y(\nu(w_2)t)\bigr),
 \end{align*}
 where we understand $w_1xL_{y}^{-1}(\nu(w_2)t)=0$ when $w_1\in\mathbb{Q}[t]$.
\end{thm}
\begin{rem}
 Taking $w_1=1$ in the above theorem, we have $\mathit{Z}_{\mathcal{F}}^t(w)=\mathit{Z}_{\mathcal{F}}^t(\nu(w))$ for $w \in \fH^1_t$, which is often called the reversal formula.
\end{rem}
To prove Theorem \ref{shF}, we need Lemmas \ref{lem:Snu} and \ref{lem:prodS_t}.
\begin{lem}\label{lem:Snu}
 In $\fH^1_t$, we have
 \begin{align*}
  S_t\circ \nu=\nu\circ S_t.
 \end{align*}
\end{lem}
\begin{proof}
 Since $S_t$ and $\nu$ are $\bQ[t]$-linear, it suffices to prove that 
 \begin{align*}
 S_t\circ \nu (z_{k_1}\cdots z_{k_r})=\nu \circ S_t(z_{k_1}\cdots z_{k_r})
 \end{align*}
 for positive integers $k_1, \ldots, k_r$ and this is obvious.
\end{proof}
\begin{lem}\label{lem:prodS_t}
 For $w_1, w_2 \in y\fH_t$, we have
 \begin{align}\label{prodS_t}
  S_t(w_1)S_t(w_2)=S_t(w_1w_2-w_1xL^{-1}_y(w_2t)),
 \end{align}
 where we understand $w_1xL_{y}^{-1}(w_2t)=0$ when $w_1\in\mathbb{Q}[t]$.
\end{lem}
\begin{proof}
 When $w_1\in \bQ[t]$, we easily see the lemma holds.
 It suffices to prove the statement for $w_1=z_{k_1}\cdots z_{k_r}$ and $w_2=z_{l_1}\cdots z_{l_s}$ with $r,s \in \bZ_{\ge1}$ and $k_1, \ldots, k_r, l_1, \ldots, l_s \in \bZ_{\ge1}$. 
Recall $\sigma_t(y)=tx+y$. Then, we see that the left-hand side of \eqref{prodS_t} is
\begin{align} \label{LHS of S_t}
S_t(w_1)S_t(w_2)
=yx^{k_1-1}\sigma_t(y)x^{k_2-1}\cdots \sigma_t(y)x^{k_r-1}yx^{l_1-1}\sigma_t(y)x^{l_2-1}\cdots \sigma_t(y)x^{l_s-1}.
\end{align}
On the other hand, the right-hand side of \eqref{prodS_t} is 
\begin{align} \label{RHS of S_t}
&S_t(w_1w_2-w_1xL^{-1}_y(w_2t))\\
&=yx^{k_1-1}\sigma_t(y)x^{k_2-1}\cdots \sigma_t(y)x^{k_r-1}\sigma_t(y)x^{l_1-1}\cdots\sigma_t(y)x^{l_s-1} \nonumber \\
&\quad -yx^{k_1-1}\sigma_t(y)x^{k_2-1}\cdots \sigma_t(y)x^{k_{r-1}-1}\sigma_t(y)x^{k_r+l_1-1}\sigma_t(y)x^{l_2-1}\cdots\sigma_t(y)x^{l_s-1}t\nonumber \\
&= yx^{k_1-1}\sigma_t(y)x^{k_2-1}\cdots \sigma_t(y)x^{k_r-1}\{\sigma_t(y)x^{l_1-1}-tx^{l_1}\}\sigma_t(y)x^{l_2-1}\cdots \sigma_t(y)x^{l_2-1}. \nonumber
\end{align}
Since $\sigma_t(y)x^{l_1-1}-tx^{l_1}=yx^{l_1-1}$, \eqref{LHS of S_t} and \eqref{RHS of S_t} coincide. This completes the proof.
\end{proof}
\begin{proof}[Proof of Theorem \ref{shF}]
We denote by $\sh$ the extension to the $\bQ[t]$-linear map $\sh \colon \fH_t \times \fH_t \rightarrow \fH_t$.
%When $w_1\in \bQ[t]$, we easily see the theorem holds.
%When $w_1 \not\in \bQ[t]$
Then, by using Lemma \ref{lem:prodS_t}, and $S_t(w_1\overset{t}{\sh}w_2)=S_t(w_1)\sh S_t(w_2) \; (w_1, w_2 \in \fH^1_t)$ (\cite[Proposition 2.1]{LQ17}), we have
 \begin{align*}
 \mathit{Z}^{t}_{\cF}(w_1\overset{t}{\sh}w_2)
 &=Z_{\cF}\bigl(S_t(w_1\overset{t}{\sh}w_2)\bigr)\\
 &=Z_{\cF}\bigl(S_t(w_1)\sh S_t(w_2)\bigr)\\
 &=Z_{\cF}\bigl(S_t(w_1)\nu(S_t(w_2))\bigr)\\
 &=Z_{\cF}\bigl(S_t(w_1)S_t(\nu(w_2))\bigr)\\
 &=Z_{\cF}\bigl(S_t(w_1\nu(w_2)-w_1xL^{-1}_y(\nu(w_2)t))\bigr)\\
 &=\mathit{Z}^{t}_{\cF}\bigl(w_1\nu(w_2)-w_1xL^{-1}_y(\nu(w_2)t)\bigr).
 \end{align*}
This completes the proof.
\end{proof}
\begin{rem}
 The shuffle relation for MZVs can be proved by the iterated integral expression of MZVs. 
 Komori--Matsumoto--Tsumura \cite[Theorem 2]{KMT11} obtained this relation using not the iterated integral but a certain partial fraction decomposition. 

 On the other hand, its counterpart for FMZVs (more precisely, the case $t=0$ in Theorem \ref{shF}) was proved by Kaneko--Zagier \cite{KZ19}. 
An alternative proof for $\cA$-MZVs was obtained by the second named author \cite[Corollary 4.1]{Ono17}, and 
for $\cS$-MZVs, several works are known (see Jarossay \cite[Th\'{e}or\`{e}me 1.7]{Jar14}, Hirose \cite[Theorem 7]{Hir19}, and Ono--Seki--Yamamoto \cite[Theorem 3.9]{OSY19}). 
\end{rem}

%%%%%%%%%%%%%%%%%%%%%%%%%%%%%%%%%%%%%%%%%%%%%%%%%%%%
\subsection{Duality relation}
We define a $\bQ$-linear isomorphism $\alpha\colon\fH \rightarrow \fH$ by interchanging $x$ and $y$, and $\alpha(1)\coloneqq 1$.
We also define a $\bQ$-linear map $\widetilde{\alpha}\colon y\fH \rightarrow y\fH$ by $\widetilde{\alpha}(yw)=y\alpha(w)$ for $w \in \fH$.
The duality relation for FMZSVs, obtained by Hoffman \cite[Theorem 4.6]{Hof15} and Jarossay \cite[Corollaire 1.12]{Jar14}, is the following equality:
\begin{align*} 
 \mathit{Z}^{1}_{\cF}(w)
 =-\mathit{Z}^{1}_{\cF}(\widetilde{\alpha}(w)) \quad (w \in y\fH).
\end{align*}
Note that we can rewrite this equality as
\begin{align*} 
 \zeta^{\star}_{\cF}(\bk)=-\zeta^{\star}_{\cF}(\bk^\vee).
\end{align*}
%%%%%%%%%%%%%%%%%%%%%%%%%%%%%%%%%%%%%%%%%%%%%%%%
It is also known that there is another expression of the above duality relation, i.e.,  
\begin{align} \label{phiF}
 Z_{\cF}(w)=Z_{\cF}(\phi(w))
\end{align}
holds (for details, see Saito \cite[Corollary 2.15]{Sai17}, for example).
Here, the map $\phi\colon\fH \rightarrow \fH$ is a $\bQ$-linear isomorphism defined by $\phi(x)=x+y$ and $\phi(y)=-y$. 
 
We introduce 
\begin{align*}
 \phi^t\coloneqq -S_{-t}\circ \phi \circ S_t,
\end{align*}
which was defined by Tanaka--Wakabayashi in \cite[eq.(3)]{TW16}. 
Then we can state the duality relation for $t$-FMZVs.
\begin{thm}[Duality relation]\label{duality}
For $w \in \fH^1_t$, we have
 \begin{align*}
 \mathit{Z}^{t}_{\cF}(w)=-\mathit{Z}^{t}_{\cF}(\phi^t(w)).
 \end{align*}
\end{thm}
\begin{proof}
Since $y\fH_tx=S_{-t}(y\fH_tx)$, the statement is equivalent to the equality
 \begin{align}\label{duality S_-t}
 \mathit{Z}^{t}_{\cF}(S_{-t}(w))=-\mathit{Z}^{t}_{\cF}(\phi^tS_{-t}(w))
 \end{align}
 for all $w \in \fH^1_t$. Since $\mathit{Z}^{t}_\cF=Z_{\cF}\circ S_t$ and $\phi^t=-S_{-t}\circ \phi \circ S_t$, 
 the equality \eqref{duality S_-t} is equivalent to \eqref{phiF}. 
 This completes the proof.
\end{proof}
\begin{rem}
 Theorem \ref{duality} with $t=0$ implies the duality relation for FMZ(S)Vs \eqref{phiF}.
\end{rem}
\begin{rem}  
There are several alternative proofs of the duality relation for FMZ(S)Vs. 
See Bachmann--Takeyama--Tasaka \cite[Theorem 2.15]{BTT18}, Seki \cite[p.28]{Sek17}, and Seki--Yamamoto \cite[Corollary 2.2]{SY19-2} for $\cA$-MZ(S)Vs, and Bachmann--Takeyama--Tasaka \cite[Corollary 2.17]{BTT18} and Hirose \cite[Theorem 8]{Hir19} for $\cS$-MZ(S)Vs.
\end{rem}

%%%%%%%%%%%%%%%%%%%%%%%%%%%%%%%%%%%%%%%%%%%%%%%%%%%%

\subsection{Derivation relation}
A derivation $\partial$ is a $\mathbb{Q}$-linear map on $\mathfrak{H}$ satisfying the Leibniz's rule $\partial(ww')=\partial(w)w'+w\partial(w') \; (w, w' \in \mathfrak{H})$. We can extend a derivation $\partial$ on $\mathfrak{H}$ to a $\mathbb{Q}[t]$-linear map on $\mathfrak{H}_t$. 

For a positive integer $l$, we define the derivation $\partial_l$ on $\mathfrak{H}_t$ by 
 \[
 \partial_l(x)=y(x+y)^{l-1}x,\quad
 \partial_l(y)=-y(x+y)^{l-1}x.
 \]
We set $\partial^t_l\coloneqq S_{-t}\circ \partial_l \circ S_t$ (see Li \cite{Li19}). 
Then we can prove that $\partial^t_l$ is a left $\widetilde{S}_t$-derivation on $\fH_t$ and we have
 \[
 \partial^t_l(x)= y(y+x-tx)^{l-1}x, \quad \partial^t_l(y)= -y(y+x-tx)^{l-1}x.
 \]
Note that this $\partial^t_l$ is not a derivation in general but an $\widetilde{S}_t$-derivation, which is a kind of a twist of the usual derivation (for the details on $\widetilde{S_t}$-derivation, see Li \cite{Li19} and Li--Qin \cite{LQ16}). We define $\bQ$-linear map $R^{-1}_x\colon\fH_t x \rightarrow \fH_t$ by $R^{-1}_x(wx)\coloneqq w \; (w \in \fH_t)$.
\begin{thm}[Derivation relation]\label{derF}
 For a positive integer $l$, we have
  \begin{align*}
  \mathit{Z}^{t}_{\mathcal{F}}(R^{-1}_x \partial^t_l(w))=0 \; (w \in y\fH_t x).
  \end{align*}
\end{thm}
\begin{proof}[Proof of Theorem \ref{derF}]
Since $S_t(y\mathfrak{H}_tx)=y\mathfrak{H}_tx$ from \cite[\S2.1]{LQ17} and the definition of $\partial^t_l=S_{-t}\circ \partial_l \circ S_t$ \cite[p.7]{Li19}, it suffices to prove the equality
\begin{align}\label{derF equiv}
\mathit{Z}^{t}_{\mathcal{F}}(R^{-1}_xS_{-t}\partial_l(w))=0
\end{align}
for all $w \in y\fH_t x$. Moreover, from the easy fact $R^{-1}_xS_{-t}=S_{-t}R^{-1}_x$ on $y\mathfrak{H}_tx$ and the relation of $\mathit{Z}^{t}_{\mathcal{F}}=Z_{\mathcal{F}}\circ S_t$, the equality \eqref{derF equiv} is equivalent to 
\[
Z_{\mathcal{F}}(R^{-1}_x\partial_l(w))=0,
\]
which is just the derivation relation for $Z_{\mathcal{F}}$ proved by the first named author in \cite{Mur17}.
\end{proof}
\begin{rem}
 The derivation relation for MZV was proved by Ihara--Kaneko--Zagier \cite[Corollary 6]{IKZ06} and Horikawa--Murahara--Oyama \cite[\S 4]{HMO18} gave another several proofs. 
 Note that Bachmann--Tanaka \cite[Theorem 1.4]{BT18}, Hirose--Murahara--Murakami \cite[\S 5.3]{HMM19} also gave another proof of the derivation relation for MZVs. Li \cite[Theorem 2.3 (6)]{Li19} proved the derivation relation for $t$-MZVs.
On the other hand, the derivation relation for FMZVs was proved by the first named author \cite[Theorem 2.1]{Mur17} and Horikawa--Murahara--Oyama \cite[\S 5]{HMO18} gave another several proofs. Note that one of their result in \cite[\S 5]{HMO18} is due to K. Ihara.
Also note that Hirose--Sato (unpublished) gave a simultaneous generalization of the derivation relations for MZVs and $\cS$-MZVs.
\end{rem}
By applying Theorem \ref{derF} for $l=1$ and $w=z_{k_1}\cdots z_{k_r}$ with $k_1, \ldots, k _r\in \bZ_{\ge1}$ and $k_r\ge2$, we obtain the Hoffman relation for $t$-FMZVs. 
\begin{cor}[Hoffman's relation] \label{hoffmanF}
 For a non-empty index $(k_1,\dots,k_r)$ with $k_r\ge2$, we have
 \begin{align*}
  &\sum_{i=1}^{r} (1+(k_i+\delta_{i,1}-2)t) 
   \zeta_\mathcal{F}^t (k_1,\dots,k_{i-1},k_{i}+1, k_{i+1},\dots,k_{r-1},k_r-1) \\
  &=\sum_{i=1}^{r} \sum_{j=2}^{k_i}
   \zeta_\mathcal{F}^t (k_1,\dots,k_{i-1},k_{i}+1-j,j,k_{i+1},\dots,k_{r-1},k_r-1) \\
  &\quad +t(1-t) \sum_{i=1}^{r-1} 
   \zeta_\mathcal{F}^t (k_1,\dots,k_{i-1},k_{i}+k_{i+1}+1,k_{i+2},\dots,k_{r-1},k_r-1).
 \end{align*}
\end{cor}
\begin{rem}
 Hoffman's relation for MZVs was proved by Hoffman \cite[Theorem 5.1]{Hof92}, which is of the same form of Corollary \ref{hoffmanF} with $t=0$.
 Another proof of Hoffman's relation for MZVs was given by Hoffman--Ohno \cite[Theorem 2.1]{HO03}. 
 Hoffman's relation for MZSVs was proved by Muneta \cite[Theorem 3.1]{Mun09} and Wakabayashi \cite[Theorem 1.1]{Wak12}. 
 These results were interpolated by Wakabayashi \cite[Corollary 1.2]{Wak17} and by Li--Qin \cite[Theorem 2.5]{LQ17} independently. 
\end{rem}

%%%%%%%%%%%%%%%%%%%%%%%%%%%%%%%%%%%%%%%%%%%%%%%%%%%%
\section*{Acknowledgement}
The authors would like to express their gratitude to the referee and the communicator for many helpful comments.

%%%%%%%%%%%%%%%%%%%%%%%%%%%%%%%%%%%%%%%%%%%%%%%%%%

\end{document}